\documentclass[12pt]{amsart}
\usepackage{a4wide}
\usepackage{amssymb}
\usepackage{amsthm}
\usepackage{amsmath}
\usepackage{amscd}
\usepackage{verbatim}
\usepackage[all]{xy}

\numberwithin{equation}{section}

\theoremstyle{plain}
\newtheorem{theorem}{Theorem}[section]
\newtheorem{corollary}[theorem]{Corollary}
\newtheorem{lemma}[theorem]{Lemma}
\newtheorem{proposition}[theorem]{Proposition}

\theoremstyle{definition}
\newtheorem{definition}[theorem]{Definition}
\newtheorem{remark}[theorem]{Remark}

\theoremstyle{remark}

\newcommand{\A}{\mathbb{A}}
\newcommand{\R}{\mathbb{R}}

\newcommand{\Q}{\mathbb{Q}}
\newcommand{\Z}{\mathbb{Z}}

\newcommand{\C}{\mathbb{C}}

\renewcommand{\H}{\mathbb{H}}

\newcommand{\D}{\mathbb{D}}

\newcommand{\zxz}[4]{\begin{pmatrix} #1 & #2 \\ #3 & #4 \end{pmatrix}}

\newcommand{\kzxz}[4]{\left(\begin{smallmatrix} #1 & #2 \\ #3 & #4\end{smallmatrix}\right) }
\newcommand{\kabcd}{\kzxz{a}{b}{c}{d}}

\newcommand{\calH}{\mathcal{H}}

\newcommand{\calK}{\mathcal{K}}
\newcommand{\calL}{\mathcal{L}}
\newcommand{\calM}{\mathcal{M}}

\newcommand{\calO}{\mathcal{O}}

\newcommand{\calW}{\mathcal{W}}

\newcommand{\frakp}{\mathfrak p}

\newcommand{\eps}{\varepsilon}
\newcommand{\bs}{\backslash}
\newcommand{\norm}{\operatorname{N}}

\newcommand{\vol}{\operatorname{vol}}
\newcommand{\tr}{\operatorname{tr}}

\newcommand{\sgn}{\operatorname{sgn}}

\newcommand{\Sl}{\operatorname{SL}}

\newcommand{\Symp}{\operatorname{Sp}}

\newcommand{\GSpin}{\operatorname{GSpin}}

\newcommand{\Mp}{\operatorname{Mp}}
\newcommand{\Orth}{\operatorname{O}}

\newcommand{\Char}{\operatorname{char}}

\newcommand{\sig}{\operatorname{sig}}

\newcommand{\SO}{\operatorname{SO}}

\newcommand{\Res}{\operatorname{Res}}

\newcommand{\Div}{\operatorname{Div}}

\newcommand{\dv}{\operatorname{div}}

\newcommand{\ch}{\operatorname{CH}}

\newcommand{\SL}{\operatorname{SL}}

\begin{document}

\title[Regularized theta lifts]{Regularized theta lifts for orthogonal groups over totally real fields}

\author[Jan H.~Bruinier]{Jan
Hendrik Bruinier}
\address{Fachbereich Mathematik,
Technische Universit\"at Darmstadt, Schlossgartenstrasse 7, D--64289
Darmstadt, Germany}
\email{bruinier@mathematik.tu-darmstadt.de}

\thanks{The author is partially supported by DFG grant BR-2163/2-1.}

\date{\today}

\begin{abstract}
We define a regularized theta lift from $\SL_2$ to orthogonal groups over totally real fields.
It takes harmonic `Whittaker forms' to automorphic Green functions and weakly holomorphic Whittaker forms to meromorphic modular forms on orthogonal groups with zeros and poles supported on special divisors, generalizing Borcherds' work on automorphic products.
To prove our results we use the spectral expansion of the lift and study its relationship with the cohomological theta lift of Kudla and Millson.
%
\end{abstract}

\maketitle


\section{Introduction}

The theory of dual reductive pairs and theta liftings provides an
important tool for the construction of automorphic forms and for
understanding the relationship between automorphic forms on different
groups. A new aspect was added to the theory by the celebrated
discovery of Harvey--Moore and Borcherds that divergent theta
integrals can often be regularized \cite{HM}, \cite{Bo2}. One can
define a regularized theta lift of (vector valued) weakly holomorphic
modular forms for $\SL_2(\Z)$ to meromorphic modular forms on
orthogonal groups associated to {\em rational\/} quadratic spaces of
signature $(n,2)$. These lifts have their zeros and poles on special
divisors (also referred to as Heegner divisors or rational quadratic
divisors). Their Fourier expansions are given by infinite products, so
called Borcherds products.  They have found various applications, for
instance in the theory of generalized Kac-Moody algebras, in the study
of moduli problems, and in the geometry and arithmetic of Shimura
varieties, see e.g.~\cite{AF}, \cite{Bo3}, \cite{Bo:CDM}, \cite{GN},
\cite{Ku:Integrals}, \cite{BBK},  \cite{BY}.

Since the work of Borcherds, there has been the question whether the
regularized theta lift can be generalized to dual reductive pairs
other than $(\SL_2,\Orth(V))$, where $V$ is a quadratic space over
$\Q$ (see \cite{Bo2}, Problem 16.4).  For instance, one would like to
define it for quadratic spaces over totally real number fields as
well.
There are two serious problems that arise. First, the special cycles
on which the lift should have its singularities are not divisors in
general, so one cannot expect that they are related to a single
meromorphic function. Second, the straightforward generalization of
weakly holomorphic elliptic modular forms, the ``input'' for the lift,
would be meromorphic Hilbert modular forms whose poles are supported
at the Baily--Borel boundary. However, by the Koecher principle, there
are no non-trivial modular forms of this type.

In the present paper, we propose a solution for the second problem.
We introduce the notion of a weakly holomorphic Whittaker form and consider a regularized theta lift of such functions.
It leads to
meromorphic modular forms with singularities
along special
divisors, generalizing Borcherds' construction of automorphic products \cite{Bo2}.
More generally, we study harmonic Whittaker forms and their
regularized theta lifts. They give rise to Green functions in the
sense of Arakelov geometry (cf.~\cite{SABK}, \cite{BKK}).

We now describe the content of this paper in more detail.  Let $F$ be
a totally real number field of degree $d$ and discriminant $D$. We
write $\hat F$ for the ring of finite adeles and $\calO_F$ for the
ring of integers of $F$. Let $(V,Q)$ be a quadratic space over $F$ of
dimension $\ell=n+2$ and assume that the signature of $V$ at the
archimedian places of $F$ is equal to
\begin{align*}
\label{eq:sighyp}
((n,2),(n+2,0),\dots,(n+2,0)).
\end{align*}
We consider the algebraic group $H=\Res_{F/\Q}\GSpin(V)$ over $\Q$
given by Weil restriction of scalars.
 Our hypothesis on the signature
 guarantees that the symmetric space $\D$ associated to $H(\R)$ carries an invariant hermitean structure and that there exist special {\em divisors}  in the sense of \cite{Ku:Duke}.

Let
$L\subset V$ be an even $\calO_F$-lattice. To simplify the exposition
we assume throughout this introduction that $L$ be unimodular. This
implies in particular that $n$ is even. The case of arbitrary even
lattices of possibly odd rank is treated in the body of the paper.
For a compact open subgroup $K\subset H(\hat \Q)$ stabilizing $L$ we
consider the Shimura variety
\begin{align*}
X_K=H(\Q)\bs (\D\times H(\hat \Q))/K.
\end{align*}
It is a quasi-projective variety of dimension $n$ defined over $F$,
see \cite{Ku:Duke}. It is projective if and only if $V$ is anisotropic
over $F$. By our assumption on the signature of $V$ this is always the
case if $d>1$.

Associated to $L$ there is a Siegel theta function
$\Theta_S(\tau,z,h)$, where $\tau=u+iv\in \H^d$, $z\in \D$, and $h\in
H(\hat \Q)$, see Section 3. In the variable $\tau$ it transforms as a
nonholomorphic Hilbert modular form of weight $(\frac{n-2}{2},
\frac{n+2}{2},\dots,\frac{n+2}{2})$ for the group
$\Gamma=\Sl_2(\calO_F)$. In the variable $(z,h)$ it is
$H(\Q)$-invariant.  For a Hilbert modular form $f$ which is
holomorphic in $\tau_1$ and antiholomorphic in $\tau_2,\dots,\tau_d$
of weight $k=:(\frac{2-n}{2}, \frac{n+2}{2},\dots,\frac{n+2}{2})$, we
would like to consider the theta integral
\[
\phi(z,h,f)= \frac{1}{\sqrt{D}}\int_{\Gamma\bs \H^d}  f(\tau)\Theta_S(\tau, z,h) (v_2\cdots v_d)^{\ell/2}\,d\mu(\tau).
\]
However, if $n\geq 2$ then there are no non-constant Hilbert modular
forms of this type with moderate growth at the cusps. In view of
Borcherds' work on regularized theta lifts for $F=\Q$ one could try to look at
the integral for Hilbert modular forms
with singularities at the cusps. But the Koecher principle implies
that there are no non-constant forms of this type when $d>1$. Now
one could further relax the assumptions on $f$ and allow singularities
on some divisor in addition to the cusps. However, then the lift tends
to behave poorly under the invariant differential operators for
$H(\R)$.

Instead, we use a different approach here. To motivate it, we briefly
revisit the case where $F=\Q$.  For simplicity we assume in this paragraph  that
$n>2$. Any weakly holomorphic modular form of weight $k=\frac{2-n}{2}$
for $\Gamma=\Sl_2(\Z)$ can be constructed as a Poincar\'e series from
a Whittaker form.  For $s\in \C$ and $v\in \R\setminus \{0\}$ we let
\[
\calM_s(v)= |v|^{-k/2} M_{\sgn(v) k/2,s/2}(|v|)\cdot e^{-v/2},
\]
where $M_{\nu,\mu}$ denotes the usual Whittaker function, and for
any positive integer $m$ we put
\[
f_m(\tau,s)= \Gamma(s+1)^{-1} \calM_s(-4\pi m v)e(-m\bar \tau),
\]
where $e(u):=e^{2\pi i u}$. A {\em harmonic Whittaker form} $f$ of
weight $k$ is a finite linear combination
of the functions $f_m(\tau,1-k)$ with $m>0$ (see Remark \ref{rem:charwhitt} for a characterization by differential equations and growth conditions).  Note that such a
function is annihilated by the hyperbolic Laplacian in weight $k$ and
has exponential growth as $v\to\infty$.  There is a differential
operator $\xi_k$ taking harmonic Whittaker forms of weight $k$ to cusp
forms of weight $2-k$. It is defined by
\begin{align*}
\xi_k(f)=2i \sum_{\gamma\in \Gamma_\infty\bs \Gamma} v^k \overline{\frac{\partial}{\partial \bar \tau} (f)}\mid_{2-k}\gamma,
\end{align*}
where $\Gamma_\infty=\{\kzxz{1}{b}{0}{1};\; b\in \Z\}$ denotes the subgroup
of translations of $\Gamma$.
We call $f$  {\em weakly holomorphic\/} if $\xi_k(f)=0$.
If $f$ is a harmonic Whittaker form of weight $k$, then the Poincar\'e
series
\begin{align}
\label{intro:ps}
\eta(f)=\sum_{\gamma\in \Gamma_\infty\bs \Gamma} f\mid_k \gamma
\end{align}
converges and defines a harmonic weak Maass form of weight $k$ for
$\Gamma$ in the sense of \cite{BF}.  The assignment $f\mapsto \eta(f)$
defines an isomorphism between the space of harmonic Whittaker forms
and the space of harmonic weak Maass forms of weight $k$.  Moreover,
$\eta(f)$ is a weakly holomorphic modular form if and only if $f$ is a
weakly holomorphic Whittaker form (see Proposition \ref{prop:wm}).
Many properties of harmonic weak Maass forms and their relationship to
weakly holomorphic modular forms and cusp forms can also be rephrased
using Whittaker forms, see Section \ref{sect:4} for details.

If $f$ is a harmonic Whittaker form and $\eta(f)$ the corresponding
weak Maass form, we can unfold the regularized theta integral
\begin{align}
\nonumber
\Phi(z,h,\eta(f))&= \int_{\Gamma\bs \H}  \eta(f)(\tau) \Theta_S(\tau, z,h) \,d\mu(\tau)\\
\label{intro:theta}
&=\int_{\Gamma_\infty\bs \H}  f(\tau) \Theta_S(\tau, z,h) \,d\mu(\tau).
\end{align}
Consequently, regularized theta
lifts of weakly holomorphic elliptic modular forms can also be viewed
as regularized theta lifts of weakly holomorphic Whittaker forms.

We now come back to the more general setup above for an arbitrary
totally real field $F$ of degree $d$.  The idea of the present paper
is that Whittaker forms have a straightforward generalization to this
situation (see Definitions \ref{def:wm} and \ref{def:wm2}).
Since they are only invariant under translations there is no Koecher
principle.  If $d>1$, then Poincar\'e series analogous to
\eqref{intro:ps} diverge wildly.  Nevertheless, in view of
\eqref{intro:theta} we consider for a Whittaker form $f$ of weight $k$
the theta integral
\[
\Phi(z,h,f)= \frac{1}{\sqrt{D}}\int_{\Gamma_\infty\bs \H^d}  f(\tau)\Theta_S(\tau, z,h) (v_2\cdots v_d)^{\ell/2}\,d\mu(\tau),
\]
where $\Gamma_\infty=\{\kzxz{1}{b}{0}{1};\; b\in \calO_F\}$ denotes the subgroup
of translations of $\Gamma=\Sl_2(\calO_F)$.  Since Whittaker forms are
exponentially increasing as $v_1\to \infty$, the integral has to be
regularized.

In Section \ref{sect:sing} we define the regularization and study its
properties. It suffices to do this for the functions
$f_{m}(\tau,1-k)$ for $m$ totally positive, see \eqref{eq:deffms}. If $\Re(s)$ is
sufficiently large, then the theta integral of $f_{m}(\tau,s)$ can be
regularized by first integrating over $u$ and afterwards over $v$, see
Definition~\ref{def:regint}. The resulting function $\Phi_m(z,h,s)$ is
an eigenfunction of the invariant Laplacian on $\D$ with a singularity
along the special divisor $Z(m)$ of discriminant $m$.  By means of ideas
of Oda and Tsuzuki \cite{OT}
we compute
the spectral expansion of $\Phi_m(z,h,s)$ and employ it to derive a
meromorphic continuation to the whole $s$-plane and a functional
equation in $s$ (see Theorems \ref{thm:spectral} and \ref{thm:smoothcont}).
We define the regularized theta integral of
$f_{m}(\tau,1-k)$ as the constant term in the Laurent expansion of
$\Phi_m(z,h,s)$ at $s=1-k$.  So for any harmonic Whittaker form
\begin{align}
\label{intro:f}
f=\sum_{m\gg 0} c(m) f_{m}(\tau,1-k)
\end{align}
of weight $k$ we
obtain a regularized theta lift $\Phi(z,h,f)$.
\begin{theorem}
  (See Theorem \ref{thm:reglift}.)  The regularized theta lift
  $\Phi(z,h,f)$ of $f$ is a logarithmic Green function in the sense of
  Arakelov geometry for the divisor
\[
Z(f)=\sum_{m\gg 0} c(m)Z(m).
\]
\end{theorem}

In Section \ref{sect:6} we investigate the relationship of the
regularized theta lift and the Kudla--Millson lift (see
e.g.~\cite{KM3}).
Recall that Kudla and Millson constructed a theta function
$\Theta_{KM}(\tau,z,h)$ which transforms in $\tau$ like a
non-holomorphic Hilbert modular form of weight
$\kappa=(\frac{n+2}{2},\dots,\frac{n+2}{2})$ and which takes values in
the closed differential forms of type $(1,1)$ on $X_K$.  For a Hilbert
cusp form $g$ of weight $\kappa$ we may consider the theta integral
\begin{align*}
\Lambda(z,h,g) = \frac{1}{\sqrt{D}} \int_{\Gamma\bs \H^d}
\overline {g(\tau)}  \Theta_{KM}(\tau,z,h) v^\kappa \,d\mu(\tau).
\end{align*}
It gives rise to a map from Hilbert cusp forms to closed harmonic
$(1,1)$-forms on $X_K$.  Using the action of various differential
operators on the Siegel and the Kudla--Millson theta kernels as in
\cite{BF}, we prove (see Theorem \ref{thm:bokm}):

\begin{theorem}
Let $f$ be a harmonic Whittaker form of weight $k$ for $\Gamma$. Then
\begin{align*}
dd^c\Phi(z,h,f)=\Lambda(z,h,\xi_k(f))-B(f)\Omega.
\end{align*}
Here $B(f)$ is a constant which is explicitly given by the Fourier
coefficients of a certain Hilbert Eisenstein series of weight $\kappa$ and $\Omega$
denotes the invariant K\"ahler form on $\D$.
\end{theorem}

If $f$ is weakly holomorphic, then $\xi_k(f)=0$.
Hence the first term on the right hand side vanishes and $\Phi(z,h,f)$
is essentially a pluriharmonic function. This can be used to prove the
main result of the present paper (Theorem \ref{thm:bop}).

\begin{theorem}
\label{intro:bop}
Let $f$ be a weakly holomorphic Whittaker form of weight $k$ for
$\Gamma$ as in \eqref{intro:f}. Assume that the coefficients $c(m)$
are integral.  Then there exists a meromorphic modular form $\Psi(z,h,f)$ for
$H(\Q)$ of level $K$ with a multiplier system of finite order such
that:
\begin{enumerate}
\item[(i)]
The weight of $\Psi$ is  $-B(f)$.
\item[(ii)]
The divisor of $\Psi$ is equal to $Z(f)$.
\item[(iii)]
The Petersson metric of $\Psi$ is given by
\[
-\log\|\Psi(z,h,f)\|_{Pet}^2 =
\Phi(z,h,f).
\]
\end{enumerate}
\end{theorem}

Up to the statement about the Fourier expansion of $\Psi$, this result
is completely analogous to Theorem 13.3 of \cite{Bo2} on the
regularized theta lift of weakly holomorphic modular forms.
When $d=1$, it
is compatible via the map $\eta$ with Borcherds' result  (and gives
a new proof of it).
Notice that
when $d>1$, the variety $X_K$ is compact and there are no Fourier
expansions.

The third assertion of Theorem \ref{intro:bop} provides a regularized
integral representation for $\log\|\Psi\|_{Pet}^2$.
In a follow-up paper \cite{BY:Ku} this is used to compute CM values and integrals of
$\log\|\Psi\|_{Pet}^2$, and more generally of the Green functions
$\Phi(z,h,f)$, extending results of \cite{Ku:Integrals}, \cite{Scho},
\cite{BY}, \cite{BK} to totally real fields. Such quantities can be interpreted as archimedian intersection pairings and therefore play an important role in
arithmetic intersection theory.

Theorem \ref{intro:bop} can be used (when $n>2$)
to show that the generating series
\[
A(\tau)= -c_1(\calM_1) + \sum_{m\gg0} Z(m)q^m,
\]
of the special divisors $Z(m)$ is a Hilbert modular form of weight
$\kappa$ with values in the first Chow group of $X_K$ (see Theorem
\ref{thm:modularity}). Here $c_1(\calM_1)$ denotes the Chern class of
the line bundle of modular forms of weight $1$.  Our proof is a
variant of the proof that Borcherds gave for $F=\Q$, see \cite{Bo3}.
This result was also proved in \cite{YZZ} using the modularity result
of Kudla--Millson \cite{KM3} for the cohomology classes of special
divisors.

In Section \ref{sect:8} we present some examples illustrating Theorem
\ref{intro:bop}.
In particular, we consider Shimura curves over totally real fields and
Shimura varieties associated to orthogonal groups of even unimodular
lattices over real quadratic fields.

I thank A. Deitmar for his help with Lemma \ref{lem:uniconv} and
G. Nebe for her help with Section \ref{sect:8.2}. Moreover, I thank E. Freitag, J. Funke, K. Ono and T. Yang for many useful conversations and comments related to this paper.

\section{Quadratic spaces and Shimura varieties}
\label{sect2}

Throughout we use the setup of \cite{Ku:Duke}, Section 1.
Let $F$ be a totally real number field of degree $d$ over $\Q$. We write $\calO_F$ for the ring of integers in $F$, and write $\partial=\partial_F$ for the different ideal of $F$.
The discriminant of $F$ is denoted by $D=\norm(\partial_F)=\#\calO_F/\partial_F$.
Let $\sigma_1,\dots,\sigma_d$ be the different embeddings of $F$ into $\R$.
We write $\A_F$ for the ring of adeles of $F$ and $\hat F$ for the subring of finite adeles. Moreover, we put $\A=\A_\Q$.

Let $(V,Q)$ be a non-degenerate quadratic space of dimension $\ell=n+2$ over $F$.
We put $V_{\sigma_i}=V\otimes_{F,\sigma_i}\R$ and identify
$V(\R)=V\otimes_\Q\R=\bigoplus_i V_{\sigma_i}$. We assume that $V$ has signature
\[((n,2),(n+2,0),\dots, (n+2,0)),
\]
that is, $V_{\sigma_1}$ has signature $(n,2)$ and $V_{\sigma_i}$ has signature $(n+2,0)$ for $i=2,\dots ,d$.
Sometimes we will also refer to the quantity
\[
\sig(V)=(n-2,n+2,\dots,n+2)\in \Z^d
\]
as the signature.
Let $\GSpin(V)$ be the `general' Spin group of $V$, that is, the group of all invertible elements $g$ in the even Clifford algebra of $V$ such that $gVg^{-1}=V$.
It is an algebraic group over $F$, and the vector representation gives rise to an exact sequence
\[
1\longrightarrow F^\times\longrightarrow  \GSpin(V)\longrightarrow  \SO(V)\longrightarrow 1.
\]
We consider the algebraic group $H=\Res_{F/\Q}\GSpin(V)$ over $\Q$ given by Weil restriction of scalars. So $H(\Q)$ can be identified with $\GSpin(V)(F)$.

We realize the hermitean symmetric space corresponding to $H$ as  the Grassmannian $\D$ of oriented negative definite $2$-dimensional subspaces of $V_{\sigma_1}$.
Note that $\D$ has two components corresponding to the two possible choices of the orientation.
The complex structure on $\D$ is most easily realized as follows. We let $V_\C=V\otimes_{F,\sigma_1} \C$ and extend the bilinear form $\C$-bilinearly to $V_\C$. The open subset
\begin{align}
\calK=\{ [Z]\in P(V_\C);\; \text{$(Z,Z)=0$ and $(Z,\bar Z)<0$}\}
\end{align}
of the zero quadric of the projective space $P(V_\C)$ of $V_\C$ is isomorphic to $\D$ by mapping $[Z]$ to the subspace $\R\Re(Z)+\R\Im(Z)\subset V_{\sigma_1}$ with the appropriate orientation.

We choose a pair $a,b\in V_{\sigma_1}$ of isotropic vectors such that $(a,b)=1$. The real quadratic space
$V_0:=V_{\sigma_1}\cap a^\perp\cap b^\perp$ has signature $(n-1,1)$. The tube domain
\begin{align}
\calH=\{ z\in V_0\otimes_\R\C;\; \text{$Q(\Im(z))<0$}\}
\end{align}
is isomorphic to $\calK$ by mapping $z\in \calH$ to the class in $P(V_\C)$ of
\begin{align}
w(z)=z+a-Q(z)b.
\end{align}
The domain $\calH$ can be viewed as a generalized complex upper  half plane.
The linear action of $H(\R)$ on $V_\C$ induces an action on $\calH$ by fractional linear transformations.
If $\gamma\in H(\R)$, we have
\begin{align}
\gamma w(z) = j(\gamma, z) w(\gamma z)
\end{align}
for an automorphy factor $j(\gamma,z)=(\gamma w(z), b)$.

The function
\begin{align}
\label{eq:pet1}
Z\mapsto -\frac{1}{2}(Z,\bar Z)=-(Y,Y)=:|Y|^2
\end{align}
on $V_\C$ defines a hermitean metric on the tautological line bundle $\calL$ over $\calK$, where $Y=\Im(Z)$.
Its first Chern form
\begin{align}
\label{eq:omega}
\Omega=-dd^c\log|Y|^2
\end{align}
is $H(\R)$-invariant and positive.  It corresponds to an invariant K\"ahler metric on $\D\cong \calK$
and gives rise to an invariant volume form $d\mu(z)=\Omega^n$.
Note that for $z\in \calH$ we have
\begin{align}
-\frac{1}{2}(w(z),\overline{w(z)})&=-(\Im(z),\Im(z)),\\
|\Im(\gamma z)|^2&= |j(\gamma,z)|^{-2}|\Im(z)|^2 .
\end{align}

For $K\subset H(\hat\Q)$ compact open we consider the Shimura variety
\begin{align}
X_K=X_{K,V}:=H(\Q)\bs (\D\times H(\hat \Q))/K.
\end{align}
It is a quasi-projective variety of dimension $n$ defined over $F$. It is projective if and only if $V$ is anisotropic over $F$.
By our assumption this is always the case when $d>1$. Let $\D^+$ be one of the two components of $\D$. The connected component of the identity $H(\R)^+$ of $H(\R)$ acts on $\D^+$. We let $H(\Q)^+=H(\Q)\cap H(\R)^+$ and write
\begin{align}
\label{decomp1}
H(\hat\Q)=\coprod_{j} H(\Q)^+ h_j K,
\end{align}
as a finite disjoint union with  $h_j\in H(\hat \Q)$. Then we have
\begin{align}
\label{decomp2}
X_K=\coprod_{j} \Gamma_j\bs \D^+,
\end{align}
where $\Gamma_j= H(\Q)^+\cap h_j K h_j^{-1}$. The different components of $X_K$ all have the same finite volume.


\subsection{Modular forms}

We define modular forms for the group $H(\Q)$ as follows.
A function $\Psi$ on $ \calH\times H(\hat\Q)$ is called a meromorphic (holomorphic) modular form of weight $w\in\Z$ and level $K$ if:
\begin{enumerate}
\item[(i)] For fixed $h\in H(\hat\Q)$ the function $\Psi(z,h)$ is meromorphic (holomorphic) in $z\in \calH$,
\item[(ii)]
$\Psi(z,hk)=\Psi(z,h)$ for all $k\in K$,
\item[(iii)]
$\Psi(\gamma z,\gamma h)= j(\gamma,z)^w\Psi(z,h)$ for all $\gamma\in H(\Q)$,
\item[(iv)]
$\Psi$ is meromorphic (holomorphic) at the boundary.
\end{enumerate}

The last condition is trivially fulfilled if $V$ is anisotropic over $F$. By the Koecher principle, it is also automatically fulfilled if the Witt rank of $V$ over $F$ (i.e.~the dimension of a maximal totally isotropic subspace over $F$) is smaller than $n$.
Using the decomposition \eqref{decomp1}, one can define functions $\Psi_j(z):=\Psi(z,h_j)$ on $\D^+$. They are classical modular forms of weight $w$ for the groups $\Gamma_j$. Consequently, the function $\Psi(z,h)$ corresponds to a tuple of classical modular forms $(\Psi_j)$.

The transformation law (iii) can be relaxed by allowing characters or multiplier systems.
By slight abuse of notation, a function $\sigma: H(\Q)\times H(\hat \Q)\to \C^\times$ is called a character for $H(\Q)$, if
$\sigma(\gamma,hk)=\sigma(\gamma,h)$ for all $k\in K$, and
$\sigma(\gamma_1\gamma_2 ,h)= \sigma(\gamma_1,\gamma_2 h)\sigma(\gamma_2,h)$ for all $\gamma_1,\gamma_2\in H(\Q)$.
Then the functions $\sigma_j(\gamma):= \sigma(\gamma,h_j)$ are homomorphisms $\Gamma_j\to \C^\times$.
A function $\Psi$ on $ \calH\times H(\hat\Q)$ is called a modular form of weight $w\in\Z$ and level $K$ with character $\sigma$, if it satisfies besides $(i)$, $(ii)$, $(iv)$
that
\begin{enumerate}
\item[(iii')]
$\Psi(\gamma z,\gamma h)= \sigma(\gamma,h) j(\gamma,z)^w\Psi(z,h)$ for all $\gamma\in H(\Q).$
\end{enumerate}
More generally, one can define modular forms of rational weight $w\in\Q$ with a multiplier system
$H(\Q)\times H(\hat \Q)\to \C^\times$, see e.g.~\cite{Br2}, Chapter 3.3.

Modular forms of weight $w$ can be viewed as sections of the
line bundle $\calM_w$ of modular forms of weight $w$.
The line bundle $\calM_1$ of modular forms of weight $1$ can be defined as the quotient
\[
H(\Q)\bs \calL \times H(\hat\Q)/K
\]
of the tautological bundle $\calL \times H(\hat \Q)$.
The $w$-th power of this bundle is the line bundle of  modular forms of weight $w$.
The hermitean metric \eqref{eq:pet1} on $\calL$ induces a metric on the bundle $\calM_w$ called the Petersson metric.
For a modular form $\Psi$ of weight $w$ it is given by
\begin{align}
\|\Psi(z,h)\|_{Pet}= |\Psi(z,h)|\cdot |y|^{w}.
\end{align}
The first Chern form of the line bundle $\calM_w$ with the Petersson metric is $w\Omega$.

\subsection{Lattices}
\label{sect:lattices}

Let $L\subset V$ be an $\calO_F$-lattice, that is, a finitely generated $\calO_F$-submodule such that $L\otimes_{\calO_F} F=V$.
We assume that $L$ is even, that is, $Q(L)\subset \partial^{-1}$.
Then $Q_\Q(x)=\tr_{F/\Q}Q(x)$ defines an even $\Z$-valued quadratic form on $L$.
Let $L'$ be the $\Z$-dual lattice of $L$ with respect to the quadratic form $Q_\Q$. The following is easily seen.

\begin{proposition}
\label{prop:lattice}
We have
\begin{enumerate}
\item $L'$ is an $\calO_F$-lattice.
\item $L\subset L'$ is a sublattice of finite index.
\item $(L,L')\subset \partial^{-1}$.
\end{enumerate}
\end{proposition}

The finite $\calO_F$-module $L'/L$ is called the discriminant group of $L$.
We write $\hat L= L\otimes_\Z \hat\Z$, where $\hat \Z=\prod_p \Z_p$. We have $L'/L\cong \hat L'/\hat L$. The lattice $L$ is called unimodular if $L'=L$.

Recall that $H(\hat \Q)$ acts on the set of lattices $M\subset V$ by $M\mapsto h M:= (h\hat M)\cap V(F)$.
This action induces an isomorphism $M'/M\to (h M)'/(h M)$, and $hM$ lies in the same genus as $M$.
Throughout we assume that the compact open subgroup $K\subset H(\hat\Q)$ fixes the lattice $L\subset V$ and acts trivially on $L'/L$.

\subsection{Special divisors}

Here we define special divisors on $X_K$ (see e.g.~\cite{Ku:Duke}, \cite{Bo2},
\cite{Br2}). They generalize Heegner divisors on
modular curves.
We follow the description in \cite{Ku:Duke}.

Let $x\in V$ be a vector of totally positive norm. We write
$V_x$ for the orthogonal complement of $x$ in $V$ and $H_x$ for the
stabilizer of $x$ in $H$. So $H_x\cong \Res_{F/\Q}\GSpin(V_x)$. The
sub-Grassmannian
\begin{align*}
\D_{x}=\{ z\in \D;\;\text{$z\perp x$}\}
\end{align*}
defines an analytic divisor on $\D$. For $h\in H(\hat\Q)$ we consider
the natural map
\begin{align*} \label{eqY4.3}
H_x(\Q)\bs \D_x\times H_x(\hat\Q)/(H_x(\hat \Q)\cap hKh^{-1})
\longrightarrow X_K,\quad (z,h_1)\mapsto (z,h_1 h).
\end{align*}
Its image defines a divisor $Z(x,h)$ on $X_K$, which is rational
over $F$.
Let $m\in F$ be totally positive, and let
$\varphi\in S(V(\hat F))^K$ be a $K$-invariant Schwartz function.
If there is an $x_0\in V(F)$ with $Q(x_0)=m$, we define the weighted cycle
\begin{align*}
Z(m,\varphi)= \sum_{h\in H_{x_0}(\hat\Q)\bs  H(\hat\Q)/K } \varphi(h^{-1} x_0) Z(x_0,h).
\end{align*}
The sum is finite, and $Z(m,\varphi)$ is a divisor on $X_K$ with complex coefficients.
If there is no $x_0\in V(F)$ with $Q(x_0)=m$, we put $Z(m,\varphi)=0$.
If $\mu\in L'/L$ is a coset, and $\chi_\mu=\Char(\mu+\hat L) \in S(V(\hat F))^K$ is the characteristic function, we briefly write
\[Z(m,\mu):=Z(m,\chi_\mu).
\]

\section{Theta functions}
\label{sect:3}
\subsection{Special Schwartz functions}
\label{sect:3.1}

We begin by recalling some properties of Schwartz functions on quadratic spaces over $\R$.
Let $(W,Q_W)$ be a real quadratic space of signature $(p,q)$ and write $S(W)$ for the space of Schwartz functions on $W$. Let $\D_W$ be the symmetric domain associated to $\SO(W)$,
realized as the Grassmannian of $q$-dimensional negative definite oriented subspaces of $W$. When $q>0$, then $\D_W$ consists of two components corresponding to the two possible choices of the orientation, and when $q=0$, then $\D_W$ is a point.

Let $\H$ be the upper complex half plane.
Let $\Mp_2(\R)$ be the two-fold metaplectic cover of $\Sl_2(\R)$  realized by the two possible choices of a holomorphic square root of the automorphy factor $c\tau+d$ for $\kabcd\in \Sl_2(\R)$.
Recall that $\Mp_2(\R)$ acts on $S(W)$ via the Weil representation $\omega_W$ associated to the standard additive character  $u\mapsto e(u)$ of $\R$.
Let $\widetilde{\SO}_2(\R)$ be the inverse image of $\SO_2(\R)$ in  $\Mp_2(\R)$, and let
\begin{align}
\label{eq:kalpha}
k_\alpha&= \zxz{\cos(\alpha)}{\sin(\alpha)}{-\sin(\alpha)}{\cos(\alpha)}\in \SO_2(\R)
\end{align}
for $\alpha \in\R$. There is a character $ \chi_{1/2}: \widetilde{\SO}_2(\R)\to \C^\times$ given by
\begin{align}
\label{eq:chi12}
\chi_{1/2}(k_\alpha, \phi)\mapsto \phi(i)^{-1}= \pm e^{i\alpha/2}.
\end{align}
We fix a base point $w_0\in \D_W$, and let
$l\in \frac{1}{2}\Z$.
Let $\varphi\in S(W)$ be a Schwartz function which is invariant under the stabilizer in $\SO(W)$ of $w_0$, and
which is an eigenfunction of weight $l$ for $\widetilde{\SO}_2(\R)$, that is,
$\omega_W(\tilde k)(\varphi)= \chi_{1/2 }^{2l}(\tilde k) \varphi$ for $\tilde k\in \widetilde{\SO}_2(\R)$.
Then we obtain a Schwartz function
\begin{align}
\label{eq:sch}
\varphi(\tau,w,\lambda)=
\phi(i)^{2l}
\big(\omega_W(\tilde g_\tau,h_w)\varphi\big)(\lambda),
\end{align}
depending on $\tau\in \H$ and $w\in \D_W$. Here $\tilde g_\tau=(g_\tau, \phi)\in \Mp_2(\R)$ with $g_\tau (i)=\tau$ and $h_w\in \SO(W)$ with $h_w w_0 = w$.

Every  $w\in \D_W$ defines an orthogonal sum decomposition $W= w^\perp\oplus w$ into a positive definite subspace $w^\perp$ and a negative definite subspace $w$. If $\lambda\in W$ we write $\lambda_{w^\perp}$ and $\lambda_w$  for the corresponding orthogonal projections. The Gaussian associated to $w$ is the Schwartz function
\[
\varphi_0^W(w,\lambda)= \exp\left(-2\pi Q_W(\lambda_{w^\perp})+2\pi Q_W(\lambda_{w})\right).
\]
It has weight $(p-q)/2$. Using the explicit formulas for the Weil representation, it is easily seen that the corresponding Schwartz function $\varphi_0^W(\tau , w,\lambda)$ as in \eqref{eq:sch} is given by
\[
\varphi_0^W(\tau , w,\lambda)= \Im(\tau)^{q/2}
e\big( Q(\lambda_{w^\perp})\tau+ Q(\lambda_{w})\bar\tau\big) .
\]


Kudla and Millson
constructed  a Schwartz form $\varphi_{KM}^W$ on $W$ taking values
in $A^q(\D_W)$, the differential $q$-forms on $\D_W$ (see e.g.~\cite{KM3}, and \cite{BF} Section 4). We will require the following properties of this form.
With respect to the natural action, we have the invariance
\begin{equation*}
\varphi_{KM}^W (w,\lambda)\in [ A^q(\D_W) \otimes S(W)]^{\SO(W)},
\end{equation*}
and $\varphi_{KM}^W(w,\lambda)$ is closed for all $\lambda \in W$.
We have $\varphi_{KM}^W (w,\lambda)= P_{KM}(w,\lambda) \varphi_{0}^W (w,\lambda)$, where
$P_{KM}(w,\lambda)$ is a polynomial  on $W$
taking values
in $A^q(\D_W)$.
Under the action of the Weil representation,  the function $\varphi_{KM}^W(\lambda)$ has weight $(p+q)/2$. We also have the corresponding Schwartz function
\begin{align}
\label{eq:kmex}
\varphi_{KM}^W(\tau,w,\lambda)=
P_{KM}(w,\sqrt{\Im(\tau)}\lambda) e\big( Q(\lambda_{w^\perp})\tau+ Q(\lambda_{w})\bar\tau\big) .
\end{align}
When $q=0$, so that $W$ is positive definite, $\varphi_{KM}^W$ is simply the Gaussian $\varphi^W_0$.

In the present paper we only need the hermitean case $q=2$, for which a convenient construction of  $\varphi_{KM}^W$ is given in \cite{Ku:Integrals}, Section 4.
We have that $\varphi_{KM}^W(0)=-\Omega$, where $\Omega$ is the negative of the invariant  K\"ahler form $\Omega$ on $\D_W$ defined in \eqref{eq:omega}.
The following relationship of $\varphi_{KM}^W$ and the Gaussian is obtained in \cite{BF}, Theorem 4.4.

\begin{proposition}
\label{prop:BF}
Assume that $q=2$ so that $\D_W$ is hermitean. We have
\[
dd^c \varphi_{0}^W(\tau,w,\lambda)=-L \varphi_{KM}^W(\tau,w,\lambda),
\]
where $L=-2i \Im(\tau)^2 \frac{\partial}{\partial\bar{\tau}}$ denotes the Maass lowering operator, and $d$ and $d^c=\frac{1}{4\pi i}(\partial -\bar \partial)$ are the usual differentials on $\D_W$.
\end{proposition}

\subsection{Discriminant forms and the Weil representation}
\label{sect:3.2}

We now come back to our global quadratic space $V$ over the totally real field $F$ as in Section~\ref{sect2}.
Let $\H$ be the upper complex half plane.
We use $\tau=(\tau_1,\dots,\tau_d)$ as a standard variable on $\H^d$ and put $u_i=\Re(\tau_i)$, $v_i=\Im(\tau_i)$. For a $d$-tuple $(w_1,\dots,w_d)$ of complex numbers, we put
$\tr(w)= \sum_{i} w_i$ and
$\norm(w)= \prod_{i} w_i$.

We view $\C^d$ as a $\R^d$-module by putting $\lambda w=(\lambda_1 w_1,\dots,\lambda_d w_d)$ for $\lambda=(\lambda_1,\dots,\lambda_d)\in \R^d$.
For $x\in F$ we briefly write $x_i=\sigma_i(x)$ and identify $x$ with its image $(x_1,\dots,x_d)\in \R^d$. The usual trace and norm of $x$ coincide with the above definitions. Moreover, the inclusion $F\to\R^d$ defines an $F$-vector space structure on $\C^d$.

We are interested in certain  vector valued modular forms for the Hilbert modular group associated to $F$. Let
$G=\Res_{F/\Q} \Sl_2$. For $g\in \Sl_2(F)\cong G(\Q)$, we briefly write $g_i=\sigma_i(g)$. So the image of $g$ in $G(\R)\cong \Sl_2(\R)^d$ is given by $(g_1,\dots,g_d)$.
The group $G(\R)$ acts on $\H^d$ by fractional linear transformations.
We denote by $\tilde G_\A$ the twofold metaplectic cover of $G(\A)$. Let  $\tilde G_\R$ be the full inverse image  in $\tilde G_\A$ of $G(\R)$.
We will frequently realize $\tilde G_\R$ as the group of pairs
\[
\left( g, \phi(\tau)\right),
\]
where $g=\kabcd\in G(\R)$ and $\phi(\tau)$ is a holomorphic function on $\H^d$ such that $\phi(\tau)^2=\norm(c\tau+d)$.
The product of $( g_1, \phi_1(\tau))$, $( g_2, \phi_2(\tau))$ is given by
\[
 ( g_1, \phi_1(\tau))\, ( g_2, \phi_2(\tau)) = (g_1g_2, \phi_1(g_2\tau)\phi_2(\tau)).
\]
Let $\tilde \Gamma$ be the full inverse image in $\tilde G_\R$ of the Hilbert modular group
\[
\Gamma=\Sl_2(\calO_F)\subset G(\R).
\]
It follows from Vaserstein's theorem that $\tilde\Gamma$ is generated by the elements
\begin{align*}
T_b &= \left( \zxz{1}{b}{0}{1 }, 1\right), \qquad b\in \calO_F,\\
S &  = \left( \zxz{0}{-1}{1}{0 }, \sqrt{\norm(\tau)}\right),\\
N &  =  \left( \zxz{1}{0}{0}{1 }, -1 \right).
\end{align*}
We also put
$Z = \left(  \kzxz{-1}{0}{0}{-1 }, i^d \right)$.
Observe that we have the relations
$(ST)^3=S^2=Z $. If $d$ is odd then $Z^2=N$, and if $d$ is even then $Z^2=N^2=1$. There are further relations corresponding to the elliptic fixed points of $\Gamma$.

Let $L\subset V$ be an $\calO_F$-lattice.
For $\mu \in L'/L$ we write $\chi_\mu=\Char(\mu+\hat L)\in S(V(\hat F))$ for the characteristic function of the coset.
Associated to the reductive dual pair $(\Sl_2, \Orth(V))$ there is a Weil representation $\omega=\omega_\psi$ of $\tilde G_\A$ on the Schwartz space $S(V(\A_F))$, where $\psi$ is the standard additive  character of $F\bs \A_F$ with $\psi_\infty (x)= e(\tr x)$ \cite{We1}.
The subspace
\[
S_L= \bigoplus_{\mu\in L'/L}  \C\chi_\mu \subset S(V(\hat F))
\]
is preserved by the action of $\widetilde{\Sl}_2(\hat \calO_F )$, the full inverse image in $\tilde G_\A$ of $\Sl_2(\hat\calO_F)\subset G(\hat \Q)$.
The canonical splitting $G(F)\to  \tilde G_\A$ defines a homomorphism
\[
\tilde \Gamma \longrightarrow   \widetilde{\Sl}_2(\hat \calO_F ), \quad \gamma \mapsto \hat \gamma,
\]
where $\hat \gamma$  is the unique element such that $\gamma \hat \gamma$ is in the image of $G(F)$.
This induces a representation $\rho_L$ of $\tilde \Gamma$ on $S_L$ by
\[
\rho_L(\gamma)\varphi = \bar \omega(\hat \gamma),\quad \gamma\in \tilde\Gamma.
\]

In terms of the above generators of $\tilde \Gamma$ the representation $\rho_L$ is given by
%
\begin{align}
\label{eq:weilt}
\rho_L(T_b)(\chi_\mu)&=e\left(\tr (Q(\mu)b) \right)\chi_\mu, \qquad b\in \calO_F,\\
\label{eq:weils}
\rho_L(S)(\chi_\mu)&=\frac{e(-\tr(\sig V)/8)}{\sqrt{|L'/L|}} \sum_{\nu\in
L'/L} e\left(-\tr(\mu,\nu)\right) \chi_\nu,\\
\label{eq:weiln}
\rho_L(N)(\chi_\mu)&=(-1)^\ell \chi_{\mu}.
\end{align}
Note that this is compatible with the conventions in \cite{Bo2}, \cite{Ku:Integrals}, \cite{Br2}, where the case $F=\Q$ is considered.
We also have the useful formula
$\rho_L(Z)(\chi_\mu)=e(-\tr(\sig V)/4) \chi_{-\mu}$.
For a totally positive unit $\eps\in \calO_F^\times$ the element $m(\eps)=(\kzxz{\eps}{0}{0}{\eps^{-1}}, 1)$ acts by
\begin{align}
\label{eq:weilme}
\rho_L(m(\eps))(\chi_\mu)&=\chi_{\eps^{-1}\mu}.
\end{align}
%
We denote the standard $\C$-bilinear pairing on $ S_L$ (the $L^2$ bilinear pairing) by
\begin{align}
\langle a,b\rangle=\sum_{\mu\in L'/L} a_\mu b_\mu
\end{align}
for $a,b \in S_L$.
The representation $\rho_L$ is unitary, that is, we have $\langle \bar\rho_L a,\rho_L b\rangle = \langle a,b\rangle$.
It is well known that $\rho_L$ factors through a finite quotient of $\tilde\Gamma$.

\subsection{Siegel theta functions}

Next, we define Siegel theta functions for the lattice $L$.
For the global quadratic space $V$ over $F$, we obtain the Gaussian on $V(\R)$ by piecing together the local data.
For $\lambda=(\lambda_1,\dots,\lambda_d)\in
V(\R)$ and $z\in \D$ we define
\[
\varphi_0(z,\lambda) = \varphi_0^{V_{\sigma_1}}(z,\lambda_1)\otimes \varphi_0^{V_{\sigma_2}}(\lambda_2)\otimes \dots \otimes \varphi_0^{V_{\sigma_d}}(\lambda_d).
\]
By our assumption on the signature of $V$, it has weight $(\frac{n-2}{2},\frac{n+2}{2}, \dots, \frac{n+2}{2})$.
We obtain the corresponding Schwartz function
\[
\varphi_0(\tau, z,\lambda) = \varphi_0^{V_{\sigma_1}}(\tau_1, z,\lambda_1)\otimes \varphi_0^{V_{\sigma_2}}(\tau_2,\lambda_2)\otimes \dots \otimes \varphi_0^{V_{\sigma_d}}(\tau_d,\lambda_d)
\]
for $\tau\in \H^d$. It can be explicitly described as follows.
We put
\[
Q(\lambda) = (Q(\lambda_1),\dots ,Q(\lambda_d))\in \R^d.
\]
If $z\in \D$ we let $\lambda_{1z}$ (respectively $\lambda_{1z^\perp}$) be the orthogonal projection of $\lambda_1$ to $z$ (respectively to $z^\perp$). Moreover, we put
\begin{align*}
\lambda_z&=(\lambda_{1z},0,\dots,0),\\
\lambda_{z^\perp} &=(\lambda_{1z^\perp},\lambda_2,\dots,\lambda_d).
\end{align*}
Hence
\[
\lambda \mapsto \tr Q(\lambda_{z^\perp})-\tr Q(\lambda_{z})
\]
defines a positive definite quadratic form on $V(\R)$, the majorant associated to $z$. It is easily seen that
\[
\varphi_0(\tau, z,\lambda) = v_1 e\big( \tr Q(\lambda_{z^\perp})\tau+ \tr Q(\lambda_{z})\bar\tau\big).
\]
In particular, this Schwartz function is non-holomorphic in $\tau_1$, but holomorphic in $\tau_2,\dots,\tau_d$.

We identify $V(F)$ with its image under the canonical embedding into $V(\R)$.
Let $ \varphi_f\in S(V(\hat F))$ be a Schwartz-Bruhat function. For $\tau\in \H^d$, $z\in \D$, and $h\in H(\hat \Q)$, we define the Siegel theta function associated to $\varphi_f$ by
\begin{align}
\label{theta2}
\theta_S(\tau,z,h;\varphi_f)  &=\sum_{\lambda\in V(F)}\varphi_f(h^{-1} \lambda)\varphi_0(\tau,z,\lambda)\\
\nonumber
&= v_1  \sum_{\lambda\in V(F)}\varphi_f(h^{-1} \lambda)
e\big( \tr Q(\lambda_{z^\perp})\tau+ \tr Q(\lambda_{z})\bar\tau\big) .
\end{align}
It satisfies the transformation formulas
\begin{align}
\label{eq:thetaH}
\theta_S(\tau,\gamma z,\gamma h;\varphi_f)&= \theta_S(\tau,z,h;\varphi_f), \quad \gamma\in H(\Q),\\
\label{eq:thetaG}
\theta_S(\gamma \tau,z,h;\varphi_f)&= (c_1\tau_1+d_1)^{-2}\phi(\tau)^\ell \theta_S( \tau,z,h;\omega(\hat \gamma^{-1})\varphi_f),\quad \gamma=(\kabcd,\phi)\in \tilde \Gamma.
\end{align}
%
%
%
%
We will also consider the $ S_L$-valued theta function
\begin{align}
\label{theta5}
\Theta_S(\tau,z,h)=\sum_{\mu\in L'/L}\theta_S(\tau,z,h;\chi_\mu)\chi_\mu.
\end{align}
It has the following transformation formula, which can be deduced from \eqref{eq:thetaG} or by applying the Poisson summation formula as in \cite{Bo2}.

\begin{theorem}
\label{thm:theta}
For $\gamma=\left( g, \phi(\tau)\right)\in \tilde\Gamma$ with $g=\kabcd$,  we have
\[
\Theta_S(\gamma\tau,z,h) = (c_1\tau_1+d_1)^{-2}\phi(\tau)^\ell\rho_L(\gamma)\Theta_S(\tau,z,h).
 \]
\hfill
$\square$
\end{theorem}

The following growth estimate will be used later. It is proved in the same way as the corresponding statement for holomorphic Hilbert modular forms.

\begin{proposition}
\label{prop:thetagrowth}
The Siegel theta function satisfies uniformly in $u$ that
\[
\theta_S(\tau,z,h;\chi_\mu)=O\left(v_1 N(v)^{-(n+2)/2}\right),\quad v_i\to 0.
\]
\end{proposition}

\subsection{Kudla--Millson theta functions}
\label{sect:3.4}

For $\lambda=(\lambda_1,\dots,\lambda_d)\in
V(\R)$ and $z\in \D$ the Kudla--Millson Schwartz form on $V(\R)$ is defined by
\[
\varphi_{KM}(z,\lambda)=\varphi_{KM}^{V_{\sigma_1}}(z,\lambda_1)\otimes \varphi_{KM}^{V_{\sigma_2}}(\lambda_2)\otimes \dots \otimes \varphi_{KM}^{V_{\sigma_d}}(\lambda_d).
\]
It has parallel weight $\ell/2=(n+2)/2$.
We also have the corresponding Schwartz form $\varphi_{KM}(\tau,z,\lambda)$ for $\tau\in \H^d$. It is non-holomorphic in $\tau_1$, but holomorphic in $\tau_2,\dots,\tau_d$.

Let $ \varphi_f\in S(V(\hat F))$ be a Schwartz-Bruhat function. For $\tau\in \H^d$, $z\in \D$, and $h\in H(\hat \Q)$,  the Kudla--Millson theta function associated to $\varphi_f$ is given by
\begin{align}
\label{thetakm2}
\theta_{KM}(\tau,z,h;\varphi_f)  &=\sum_{\lambda\in V(F)}\varphi_f(h^{-1} \lambda)\varphi_{KM}(\tau,z,\lambda),
\end{align}
see \cite{KM1}, \cite{KM2}, \cite{KM3} for details.
Its geometric significance lies in the fact  that the Fourier coefficients
with totally positive index $m$ are Poincar\'e dual forms for the cycles $Z(m,\varphi_f)$ (however, we do not need that here).

The characteristic functions of the cosets of $L$ can be used  to define the  $ S_L$-valued valued Kudla--Millson theta function
\begin{align}
\label{thetakm5}
\Theta_{KM}(\tau,z,h)=\sum_{\mu\in L'/L}\theta_{KM}(\tau,z,h;\chi_\mu)\chi_\mu.
\end{align}
In the variable $(z,h)$ it defines a closed $2$-form on $X_K$. In $\tau$ it satisfies the transformation formula
\[
\Theta_{KM}(\gamma\tau,z,h) = \phi(\tau)^{\ell}\rho_L(\gamma)\Theta_{KM}(\tau,z,h),
 \]
for $\gamma=\left( g, \phi(\tau)\right)\in \tilde\Gamma$.
Hence it is a non-holomorphic vector valued Hilbert modular form for $\tilde \Gamma$ of parallel weight $\ell/2$
with representation $\rho_L$. Since it has moderate growth at the cusps, it can be integrated against cusp forms.

\begin{proposition}
\label{prop:thetakmgrowth}
The Kudla--Millson theta function satisfies uniformly in $u$ that
\[
\theta_{KM}(\tau,z,h;\chi_\mu)=O\left(N(v)^{-(n+2)/2}\right),\quad v_i\to 0.
\]
\end{proposition}

\section{Whittaker forms}
\label{sect:4}

Following Harvey, Moore, and Borcherds (see \cite{HM} and \cite{Bo2}) we would like to construct automorphic forms on $X_K$ with singularities along special cycles by integrating weakly holomorphic modular forms or weak Maass forms against the Siegel theta function.
Unfortunately,  because of the Koecher principle, there are no such automorphic forms with singularities at the cusps when $d>1$.

Here we overcome this problem by viewing weak Maass forms as {\em formal\/} Poincar\'e series ignoring the issue of convergence. Then the theta integral can be formally unfolded leading to an integral over $\tilde \Gamma_\infty \bs \H^d$, where
\begin{align}
\tilde\Gamma_\infty:=\{T_b\cdot (1,\pm 1);\quad b\in \calO_F\}
\end{align}
is the subgroup of translations of $\tilde \Gamma$.
It turns out that such integrals still make sense when they are suitably regularized. It is natural to consider them for translation invariant functions which are eigenfunctions for the Laplacians. This leads to the definition of Whittaker forms below. Moreover, we define weakly holomorphic Whittaker forms, which serve as substitutes of weakly holomorphic modular forms.
We derive some properties of these functions which will be important later.

We begin by fixing some notation.
Let $k=(k_1,\dots,k_d)\in (\frac{1}{2}\Z)^d$ be a weight. Throughout we assume that $k\equiv (\frac{\ell}{2},\dots,\frac{\ell}{2})\pmod{\Z^d}$.
%
We define a Petersson slash operator in weight $k$ for the representation $\rho_L$ on functions $f:\H^d\to S_L$ by
\[
(f\mid_{k,\rho_L}(g,\phi)) (\tau)= (c_1\tau_1+d_1)^{-k_1+\ell/2}\cdots (c_d\tau_d+d_d)^{-k_d+\ell/2} \phi(\tau)^{-\ell} \rho_L(g,\phi)^{-1}f(g\tau),
\]
where $(g,\phi)\in \tilde G_\R$ and $g=\kabcd$.
The Petersson slash operator for the dual representation  $\bar\rho_L$ is defined analogously. We write $S_{k,\rho_L}$ for the space of vector valued Hilbert cusp forms of weight $k$ for $\tilde \Gamma$ with representation $\rho_L$.

We have the usual hyperbolic Laplace operators in weight $k$ acting on smooth functions on $\H^d$. They are given by
\begin{align}
\label{defdelta}
\Delta_k^{(j)} = -v_j^2\left( \frac{\partial^2}{\partial u_j^2}+ \frac{\partial^2}{\partial v_j^2}\right) + ik_j v_j\left( \frac{\partial}{\partial u_j}+i \frac{\partial}{\partial v_j}\right)
\end{align}
for $j=1,\dots,d$.
Moreover, we have the Maass lowering and raising operators
\begin{align*}
R_k^{(j)}  &=2i\frac{\partial}{\partial\tau_j} + k_j v_j^{-1},\\
L_k^{(j)}  &= -2i v_j^2 \frac{\partial}{\partial\bar{\tau_j}}.
\end{align*}
The raising operator $R_k^{(j)}$ raises the weight of an automorphic form in the $j$-th component by $2$ while $L_k^{(j)}$ lowers it by $2$.
The Laplacian $\Delta_k^{(j)}$ satisfies the identity
\begin{align*}
-\Delta_k^{(j)} = L^{(j)}_{k+2} R^{(j)}_k +k_j = R^{(j)}_{k-2} L^{(j)}_{k}.
\end{align*}

We consider functions which are invariant under the group of translations
$\tilde\Gamma_\infty$.
For $s\in \C$ we let $A_{k,\bar\rho_L}(s)$ be the space of smooth functions $f:\H^d\to  S_L$ satisfying:
\begin{enumerate}
\item
$f(T_b\tau) =  \bar \rho_L(T_b) f(\tau),\qquad b\in \calO_F$,
\item
$\Delta_k^{(1)} f =\frac{1}{4}(k_1-1+s)(k_1-1-s)f$,
\item
$f$ is antiholomorphic in $\tau_j$ for $j=2,\dots,d$.
\end{enumerate}

To describe the Fourier expansion of such a function we recall some properties of Whittaker functions, see  \cite{AS} Chap.~13 pp.~189 or \cite{B1} Vol.~I Chap.~6 p.~264.
%
Kummer's confluent hypergeometric function is defined by
\begin{align} \label{kummer1}
 M(a,b, z) = \sum_{n=0}^\infty \frac{(a)_n }{(b)_n} \frac{z^n}{n!},
\end{align}
where
$ (a)_n = \Gamma(a+n)/\Gamma(a)$ and $(a)_0=1$.
The Whittaker functions are defined by
\begin{align}
\label{eq:m1}
M_{\nu,\mu}(z)&= e^{-z/2} z^{1/2+\mu} M(1/2+\mu-\nu, 1+2\mu, z),\\
\label{eq:w1}
W_{\nu,\mu}(z)&= \frac{\Gamma(-2\mu)}{\Gamma(1/2-\mu-\nu)}M_{\nu,\mu}(z)
+\frac{\Gamma(2\mu)}{\Gamma(1/2+\mu-\nu)}M_{\nu,-\mu}(z).
\end{align}
They are linearly independent solutions of the Whittaker differential equation.
The $M$-Whittaker function has the asymptotic behavior
\begin{align}\label{Masy1}
M_{\nu,\,\mu}(z) &= z^{\mu+1/2}(1+O(z)), \qquad z\to 0,\\
\label{Masy2}
M_{\nu,\,\mu}(z) &= \frac{\Gamma(1+2\mu)}{\Gamma(\mu-\nu+1/2)}e^{z/2}z^{-\nu} ( 1+O(z^{-1})) ,\qquad z\to \infty,
\end{align}
while $W_{\nu,\,\mu}(z)$ is exponentially decreasing for real $z\to \infty$ and behaves like a constant times $z^{-\mu+1/2}$ as $z\to 0$.

For convenience we put for $s\in\C$ and $v_1\in\R$:
\begin{align}
\label{calM}
\calM_s(v_1)&=
|v_1|^{-k_1/2} M_{ \sgn(v_1)k_1/2,\,s/2}(|v_1|)
\cdot
e^{-v_1/2},\\
\label{calW}
\calW_s(v_1)&= |v_1|^{-k_1/2} W_{\sgn(v_1)k_1/2 ,\,s/2}(|v_1|)\cdot
e^{-v_1/2}.
\end{align}
%
The functions $\calM_s(v_1)$ and  $\calW_s(v_1)$ are holomorphic in $s$. Later we will be interested in their special value at
\begin{align}
s_0=1-k_1.
\end{align}
We have
\begin{align}
\label{Mspecial}
\calM_{s_0}(v_1)&= (-\sgn(v_1))^{k_1-1}\cdot e^{-v_1} \big(\Gamma(2-k_1)-(1-k_1)\Gamma(1-k_1,-v_1)\big),
\\
\label{Wspecial}
\calW_{s_0}(v_1)&= \begin{cases} e^{-v_1}
,&\text{$v_1>0$,}\\
e^{-v_1}\cdot  \Gamma(1-k_1,-v_1)
,&\text{$v_1<0$,}
\end{cases}
\end{align}
where $\Gamma(a,z)$ denotes the incomplete Gamma function.
Any $f\in A_{k,\bar\rho_L}(s)$ has a Fourier expansion of the form
\begin{align*}
f(\tau)&=\sum_{\substack{\mu\in L'/L\\ Q(\mu)\in \partial_F^{-1}}} \big(a(0,\mu,s)v_1^{(1-k_1-s)/2}+b(0,\mu,s)v_1^{(1-k_1+s)/2 }\big)\chi_\mu\\
&\phantom{=}{}+\sum_{\mu\in L'/L} \sum_{\substack{m\in \partial_F^{-1}-Q(\mu)\\m\neq 0}} \big(a(m,\mu,s)\calW_s(4\pi m_1 v_1)+b(m,\mu,s)\calM_s(4\pi m_1 v_1)\big)e(\tr(m\bar \tau))\chi_\mu.
\end{align*}
Special elements of $A_{k,\bar\rho_L}(s)$ are the functions
\begin{align}
\label{eq:deffms}
f_{m,\mu}(\tau,s):=
C(m,k,s)\calM_s(-4\pi m_1 v_1)e(-\tr (m\bar \tau) ) \chi_\mu 
\end{align}
for $\mu \in L'/L$, $m\in \partial^{-1}+Q(\mu)$, and $m\gg 0$.
Here $C(m,k,s)$ denotes the normalizing factor
\begin{align}
\label{eq:defcms}
C(m,k,s):=\frac{(4\pi m_2)^{k_2-1}\cdots (4\pi m_d)^{k_d-1}}{\Gamma(s+1)\Gamma(k_2-1)\cdots \Gamma(k_d-1)},
\end{align}
which turns out to be convenient later (for instance in Proposition \ref{proppair}).

\begin{definition}
\label{def:wm}
A {\em Whittaker form} of weight $k$ and parameter $s$ (for $\tilde \Gamma$ and $\bar \rho_L$) is a finite linear combination of the functions $f_{m,\mu}(\tau,s)$ for $\mu\in L'/L$, $m\in  \partial_F^{-1}+Q(\mu)$, and $m\gg 0$.
A {\em harmonic Whittaker form} is a Whittaker form with parameter $s_0$.
We denote the $\C$-vector space of such harmonic Whittaker forms by $H_{k,\bar\rho_L}$.
\end{definition}



So a harmonic Whittaker form is a finite linear combination of the functions
\[
f_{m,\mu}(\tau):=f_{m,\mu}(\tau,s_0)
\]
for $\mu\in L'/L$, $m\in  \partial_F^{-1}+Q(\mu)$, and $m\gg 0$.
Explicitly we have
\begin{align}
\label{eq:fs0}
f_{m,\mu}(\tau)
&= C(m,k,s_0)\Gamma(2-k_1)
\left(1-\frac{\Gamma(1-k_1,4\pi m_1 v_1)}{\Gamma(1-k_1)}\right)e^{4\pi m_1 v_1}e(\tr(-m\bar\tau))\chi_\mu .
%
\end{align}
We define the {\em dual weight} for $k$ by $\kappa= (2-k_1,k_2,\dots,k_d)$.
We consider the differential operator $\delta_k$ on functions $f:\H^d\to  S_L$ given by
\begin{align}
\label{def:delta}
\delta_k(f)=v_1^{k_1-2} \overline{L^{(1)}_k f(\tau)}.
\end{align}
If $f\in H_{k,\bar\rho_L}$, then $\delta_k(f)$ is a holomorphic function satisfying
$f(T_b\tau) =  \rho_L(T_b) f(\tau)$ for all $b\in \calO_F$.
In particular, we have
\begin{align}
\label{eq:delta}
\delta_k(f_{m,\mu})(\tau)= \frac{(4\pi m_1)^{\kappa_1-1}\cdots (4\pi m_d)^{\kappa_d-1}}{\Gamma(\kappa_1-1)\cdots \Gamma(\kappa_d-1)}
e(\tr(m\tau))\chi_\mu.
\end{align}

We have $\Delta_k^{(1)} f =0$ for a harmonic Whittaker form $f$.
The functions $f_{m,\mu}(\tau,-s_0)$ have eigenvalue zero for this Laplacian
as well, but they are {\em not\/} harmonic Whittaker forms in the sense of our definition.
The point of the definition is to prescribe a particular growth as $v_j\to \infty$ and $v_j\to 0$ such that there is a smooth relationship with Hilbert cusp forms (see below) and such that the theta lift of a harmonic Whittaker form has singularities along special divisors (see Section \ref{sect:sing}). More precisely, we have the following characterization:

\begin{remark}
\label{rem:charwhitt}
Assume that $k_1<1$ (which will be the case in all later applications).
Then $H_{k,\bar\rho_L}$ is the space of functions in $A_{k,\bar\rho_L}(s_0)$ satisfying:
\begin{enumerate}
\item[(4)]
only finitely many Fourier coefficients are non-zero,
\item[(5)]
there is a totally positive $\eps\in F$ such that $\delta_k(f)(\tau)=O(e^{-\tr (\eps v)})$ for $v_j\to \infty$,
\item[(6)]
we have $f(\tau)=O(v_1^{s_0})$, for $v_1\to 0$.
\end{enumerate}
\end{remark}

For the rest of this section we assume that $\kappa_j\geq 3/2 $ for $j=1,\dots,d$.
We define an operator $\xi_k:H_{k,\bar\rho_L}\to S_{\kappa,\rho_L}$ by
\begin{align}
\label{def:xi}
\xi_k(f)=\sum_{\gamma\in \tilde\Gamma_\infty\bs \tilde\Gamma} \delta_k(f)\mid_{\kappa,\rho_L}\gamma.
\end{align}
When $\kappa_j>2$ for $j=1,\dots,d$, the Poincar\'e series on the right hand side converges normally by a standard estimate (see e.g.~\cite{Ga}, Chapter 1.13) and defines a holomorphic cusp form. When $\kappa_j\geq 3/2$ we define the Poincar\'e series using ``Hecke summation'' as the value at $s'=0$ of the holomorphic continuation in $s'$ of
\begin{align}
\label{def:xi2}
\sum_{\gamma\in \tilde\Gamma_\infty\bs \tilde\Gamma} \delta_k(f)\norm(v)^{s'}\mid_{\kappa,\rho_L}\gamma.
\end{align}


\begin{proposition}
\label{prop:xisurj}
Assume that $\kappa_j\geq 3/2$ for $j=1,\dots,d$. The map $\xi_k:H_{k,\bar\rho_L}\to S_{\kappa,\rho_L}$ is surjective.
\end{proposition}

\begin{proof}
The assertion follows from \eqref{eq:delta} and the fact that $S_{\kappa,\rho_L}$ is generated by Poincar\'e series.
\end{proof}

\begin{definition}
\label{def:wm2}
A Whittaker form $f$ is called {\em weakly holomorphic} if it is harmonic and satisfies $\xi_k(f)=0$.
We denote by $M^!_{k,\bar\rho_L}$ the subspace of weakly holomorphic Whittaker forms in $H_{k,\bar\rho_L}$.
\end{definition}

Note that a weakly holomorphic Whittaker form is {\em not\/} holomorphic as a function on $\H^d$.
It is rather holomorphic in a weak distribution sense.
In view of Proposition \ref{prop:xisurj}, we have the exact sequence
\begin{align}
\label{ex-sequ}
\xymatrix{ 0\ar[r]& M^!_{k,\bar\rho_L} \ar[r]& H_{k,\bar\rho_L}
\ar[r]^{\xi_k}& S_{\kappa,\rho_L} \ar[r] & 0 }.
\end{align}

Recall that the Petersson scalar product of $f,g\in S_{\kappa,\rho_L}$ is given by
\begin{align}
(f, g)_{Pet}= \frac{1}{\sqrt{D}}\int_{\tilde \Gamma\bs \H^d} \langle f,\bar g \rangle v^\kappa \,
d\mu(\tau),
\end{align}
where $d\mu (\tau)=\frac{du_1\,dv_1}{v_1^2}\dots \frac{du_d\,dv_d}{v_d^2}$ is the invariant measure on $\H^d$, and $v^\kappa$ is understood in multi-index notation.
We define a bilinear pairing between the spaces $S_{\kappa,\rho_L}$  and $H_{k,\bar\rho_L}$ by putting
\begin{equation}\label{defpair}
\{g,f\}=\big( g,\, \xi_k(f)\big)_{Pet}
\end{equation}
for $g\in S_{\kappa,\rho_L}$ and $f\in H_{k,\bar\rho_L}$. The pairing vanishes when $f$ is weakly holomorphic. Because of Proposition~\ref{prop:xisurj} the induced pairing between $S_{\kappa,\rho_L}$ and $ H_{k,\bar\rho_L}/M^!_{k,\bar\rho_L}$
is non-degenerate. So an $f\in H_{k,\bar\rho_L}$ is weakly holomorphic, if and only if $\{g,f\}=0$ for all $g\in S_{\kappa,\rho_L}$.

\begin{proposition}\label{proppair}
Let $g\in S_{\kappa,\rho_L}$ with Fourier expansion $g=\sum_{n,\nu}b(n,\nu) e(\tr(n\tau))\chi_\nu$, and let
\[
f=\sum_{\mu\in L'/L} \sum_{m\gg 0} c(m,\mu)f_{m,\mu}(\tau)\in H_{k,\bar\rho_L}.
\]
Then the pairing
of $g$ and $f$ is equal to
\begin{equation}
\label{pairalt}
\{g,f\}= \sum_{\mu\in L'/L} \sum_{m\gg 0}  c(m,\mu) b(m,\mu).
\end{equation}
\end{proposition}

\begin{proof}
This follows from \eqref{eq:delta} and \eqref{def:xi} using the formula for the Petersson scalar product of Poincar\'e series with cusp forms, see e.g.~
\cite[Section 2]{Luo}.
\end{proof}


\subsection{Whittaker forms and weak Maass forms}

We end this section by explaining the relationship between Whittaker forms and weak Maass forms as defined in \cite{BF}.
Assume that $d=1$ so that $F=\Q$. Then there is no Koecher principle and there are nontrivial weak Maass forms.
For simplicity we also assume that $k=k_1<0$.

A smooth function $f:\H\to
 S_L$ is called a {\em harmonic weak Maass form} (of weight $k$ with
respect to $\tilde \Gamma$ and $\bar \rho_L$) if it satisfies:
\begin{enumerate}
\item[(i)]
$f \mid_{k,\bar\rho_L} \gamma= f$
for all $\gamma\in \tilde \Gamma$,
\item[(ii)]
$\Delta_k f=0$, where
$\Delta_k$ is the weight $k$ Laplacian,
\item[(iii)]
there is a $ S_L$-valued Fourier polynomial
\[
P_f(\tau)=\sum_{\mu\in L'/L}\sum_{m\geq 0} c^+(m,\mu) q^{-m}\chi_\mu
\]
such that $f(\tau)-P_f(\tau)=O(e^{-\eps v})$ as $v\to \infty$ for
some $\eps>0$.
%
\end{enumerate}
The Fourier polynomial $P_f$  is called the {\em principal part} of
$f$. Here we denote the vector space of these harmonic weak Maass forms
by  $\calH_{k,\bar\rho_L}$.
Any weakly holomorphic modular form is a harmonic weak Maass form.
We denote the subspace of weakly holomorphic modular forms
by
$\calM^!_{k,\bar\rho_L}$.

\begin{proposition}
\label{prop:wm}
If $f\in H_{k,\bar\rho_L}$ is a harmonic Whittaker form, then
\[
\eta(f)=\sum_{\gamma\in \tilde\Gamma_\infty\bs \tilde\Gamma} f\mid_{k,\bar\rho_L}\gamma
\]
converges and defines an element of  $\calH_{k,\bar\rho_L}$.
The map $\eta:H_{k,\bar\rho_L}\to \calH_{k,\bar\rho_L}$ defined by  $f\mapsto \eta(f)$ is an isomorphism. Its inverse is given by mapping a harmonic weak Maass form $g$ with principal part  $P_g=\sum_{\mu}\sum_{m\geq 0} c^+(m,\mu) q^{-m}\chi_\mu$ to the harmonic Whittaker form
\[
f=\sum_{\mu}\sum_{m>0} c^+(m,\mu) f_{m,\mu}(\tau).
\]
The restriction of $\eta$ induces an isomorphism $M^!_{k,\bar\rho_L}\to \calM^!_{k,\bar\rho_L}$.
\end{proposition}

The proof of the proposition follows from well known properties of non-holomorphic Poincar\'e series, see e.g.~\cite{He} or \cite{Br2}, Chapter 1.
Moreover, the operator $\xi_k$ on harmonic Whittaker forms is compatible with the corresponding operator on harmonic weak Maass forms of \cite{BF}.

\section{The theta lift}
\label{sect:sing}

Here we define the regularized theta lift of Whittaker forms with parameter $s$. The images of the lift are automorphic Green functions for special divisors.
We first do this when $\Re(s)$ is sufficiently large.
For the general case we use meromorphic continuation in $s$, which can be obtained by means of spectral theory for $X_K$.

We use the setup of the previous sections. Throughout we assume that the compact open subgroup $K\subset H(\hat \Q)$ fixes the lattice $L\subset V$ and acts trivially on $L'/L$.
Let
$\Delta_\D$ be the Laplace operator
induced by the the Casimir element of
the Lie algebra of $\SO(V_{\sigma_1})$ (or by the invariant metric on $\D$).
We normalize it as in
\cite{Br2} (4.1).

Let $f$ be a Whittaker form of weight
\begin{align}
\label{eq:k}
k=\left(\frac{2-n}{2}, \frac{2+n}{2},\dots,\frac{2+n}{2}\right)
\end{align}
with parameter $s$ for $\tilde\Gamma$ and $\bar\rho_L$ (see Definition \ref{def:wm}).
We assume that $s>s_0:=1-k_1=n/2$.
Recall from Theorem \ref{thm:theta} that the Siegel theta function $\Theta_S(\tau, z,h)$ associated to the lattice $L$ has weight $\left(\frac{n-2}{2}, \frac{2+n}{2},\dots,\frac{2+n}{2}\right)$ and representation $\rho_L$. Consequently, the pairing
\[
\langle f(\tau), \Theta_S(\tau, z,h)\rangle(v_2\cdots v_d)^{\ell/2}
\]
is a scalar valued $\tilde\Gamma_\infty$-invariant function in $\tau$. We want to consider the theta integral
\[
\Phi(z,h,f)= \frac{1}{\sqrt{D}}\int_{\tilde \Gamma_\infty\bs \H^d}\langle f(\tau), \Theta_S(\tau, z,h)\rangle (v_2\cdots v_d)^{\ell/2}\,d\mu(\tau).
\]
Because of the exponential growth of $f$, the integral does not converge. Similarly as in \cite{Bo2} and \cite{Br2}, we regularize it by prescribing the order of integration, namely we first integrate over $u$ and then over $v$.
In the notation the regularization is indicated by the superscript ``reg'' at the integral.

\begin{definition}
\label{def:regint}
The regularized theta lift of $f$ is defined as
\begin{align*}
\Phi(z,h,f)&=
\frac{1}{\sqrt{D}}\int_{\tilde \Gamma_\infty\bs \H^d}^{reg}\langle f(\tau), \Theta_S(\tau, z,h)\rangle (v_2\cdots v_d)^{\ell/2}\,d\mu(\tau)\\
&=\frac{1}{\sqrt{D}}\int_{v\in (\R_{>0})^d}\left(\int_{u\in \calO_F\bs \R^d}
\langle f(\tau), \Theta_S(\tau, z,h)\rangle \,du\right) (v_2\cdots v_d)^{\ell/2}
\,\frac{dv}{\norm(v)^{2}}.
\end{align*}
\end{definition}

\begin{theorem}
\label{thm:lift1}
Let $f$ be a Whittaker form of parameter $s$ as above. If
$\Re(s)>s_0+2$,
then the regularized theta integral
converges for $(z,h)$ outside a subset of $X_K$ of measure zero. It defines an integrable function on $X_K$.
\end{theorem}

To prove this Theorem, it suffices to consider for any $\mu\in L'/L$ and any totally positive $m\in  \partial_F^{-1}+Q(\mu)$ the theta integral
\begin{align}
\Phi_{m,\mu}(z,h,s)&:=\Phi(z,h,f_{m,\mu}(\cdot,s))
\end{align}
of the special Whittaker forms $f_{m,\mu}(\tau,s)$, see \eqref{eq:deffms}. We call $\Phi_{m,\mu}(z,h,s)$ the {\em automorphic Green function} of the divisor $Z(m,\mu)$.
Later we will show that its regularized value at $s_0$ gives rise to a subharmonic Arakelov Green function for $Z(m,\mu)$ in the sense of \cite{SABK}.
Theorem \ref{thm:lift1} will follow from Theorem \ref{thm:lift2} below.

Let $F(a,b,c; z)$ denote the Gauss hypergeometric function
\begin{align}\label{gauss1}
 F(a,b,c; z) = \sum_{n=0}^\infty \frac{(a)_n (b)_n}{(c)_n} \frac{z^n}{n!},
\end{align}
see e.g.~\cite{AS} Chap.~15 or \cite{B1} Vol.~I Chap.~2.
The circle of convergence of the series (\ref{gauss1}) is the unit circle $|z|=1$.
For $\lambda\in V(\R)$, $z\in \D$, and $s\in \C$ we put
\begin{align}
\label{eq:phi}
\phi(\lambda,z,s):=\frac{\Gamma(\frac{s}{2}+\frac{n}{4})}{\Gamma(s+1)}\left(\frac{Q(\lambda_1)}{Q(\lambda_{1z^\perp})}\right)^{\frac{s}{2}+\frac{n}{4}}
 F\left(\frac{s}{2}+\frac{n}{4},\, \frac{s}{2}-\frac{n}{4}+1, \, s+1; \, \frac{Q(\lambda_1)}{Q(\lambda_{1z^\perp})} \right).
\end{align}
This function has the invariance property $\phi(\lambda,z,s)=\phi(h\lambda,hz,s)$ for $h\in H(\R)$. Notice that it is closely related to the secondary spherical function $\phi_{s}^{(2)}(h)$ on $\SO(n,2)$ considered in \cite{OT}.
More precisely, for a totally positive $\lambda\in V$,  we fix a base point $z_0\in \D$ such that $z_0\perp\lambda$.
As in the proof of Theorem 4.7 of \cite{BK} it is easily verified that
\begin{align}
\label{eq:compot}
\phi(\lambda,hz_0,s)=\frac{-2}{\Gamma(\frac{s}{2}-\frac{s_0}{2}+1)} \phi_{s}^{(2)}(h)
\end{align}
for $h\in \SO(V_{\sigma_1})\cong \SO(n,2)$.
It follows from \cite{OT}, Proposition 2.4.2, that $\phi(\lambda,z,s)$ satisfies the differential equation
\begin{align}
\label{eq:diffphi}
\Delta_\D\phi(\lambda,z,s)=\frac{1}{8}(s^2-s_0^2)\phi(\lambda,z,s).
\end{align}

\begin{theorem}
\label{thm:lift2}
Let $\mu\in L'/L$, $m\in  \partial_F^{-1}+Q(\mu)$, and $m\gg 0$. If $\Re(s)>s_0+2$, the regularized theta integral of $f_{m,\mu}(\tau,s)$ converges and is equal to
\begin{align}
\label{eq:ser}
\Phi_{m,\mu}(z,h,s)
&=
\sum_{\substack{\lambda\in h(\mu+L)\\Q(\lambda)=m}}\phi(\lambda,z,s).
%
\end{align}
The sum converges for $\Re(s)>s_0$ and $(z,h)$ outside a subset of $X_K$ of measure $0$. It defines an integrable function on $X_K$.
\end{theorem}

\begin{proof}
We first compute the theta integral $\Phi_{m,\mu}(z,h,s)$ formally. Afterwards we show the convergence of the infinite series representation in the statement of the theorem.
The interchange of integration and summation in the computation of the theta integral is then justified a posteriori by the theorem of monotone convergence.

Inserting the definitions and carrying out the integration over $u$, we obtain
\begin{align*}
\Phi_{m,\mu}(z,h,s)&=C(m,k,s)\\
&\phantom{=}{}\times
\int_{ (\R_{>0})^d}
\sum_{\substack{\lambda\in h(\mu+L)\\Q(\lambda)=m}}
\calM_s(-4\pi m_1 v_1)
e^{-4\pi Q(\lambda_{1z^\perp})v_1} v_1^{-1}(v_2\cdots v_d)^{\ell/2-2}\,dv.
\end{align*}
Here we have also used that $\vol(\calO_F\bs \R^d)=\sqrt{D}$.
In view of Proposition \ref{prop:thetagrowth} and \eqref{Masy1}, the integral converges for $\Re(s)>s_0+2$ and $(z,h)\in X_K\bs Z(m,\mu)$.

We put in the definition of $\calM_s$ and interchange integration and summation. We  find
\begin{align*}
\Phi_{m,\mu}(z,h,s)=C(m,k,s)
\sum_{\substack{\lambda\in h(\mu+L)\\Q(\lambda)=m}} &
\int_{0}^\infty
 \frac{M_{-k_1/2,\,s/2}(4\pi m_1 v_1)}{(4\pi m_1 v_1)^{k_1/2}}
e^{-2\pi m_1 v_1+4\pi Q(\lambda_{1z})v_1}\,\frac{dv_1}{v_1}\\
\times &
\int_{0}^\infty
e^{-4\pi m_2 v_2} v_2^{k_2-1}\,\frac{dv_2}{v_2}\cdots
\int_{0}^\infty
e^{-4\pi m_d v_d} v_d^{k_d-1}\,\frac{dv_d}{v_d}.
\end{align*}
Inserting the value of $C(m,k,s)$ and carrying out the integration over $v_2,\dots,v_d$, we obtain
\begin{align*}
\Phi_{m,\mu}(z,h,s)&=\frac{1}{\Gamma(s+1)(4\pi m_1)^{k_1/2}}\\
&\phantom{=}{}\times
\sum_{\substack{\lambda\in h(\mu+L)\\Q(\lambda)=m}}
\int_{0}^\infty
M_{-k_1/2,\,s/2}(4\pi m_1 v_1)
e^{-2\pi m_1 v_1+4\pi Q(\lambda_{1z})v_1} v_1^{-k_1/2}\,\frac{dv_1}{v_1}.
\end{align*}
The latter integral is a Laplace transform. It is equal to
\[
 \frac{(4\pi m_1)^{s/2+1/2}\Gamma(\frac{s}{2}+\frac{n}{4})}{(4\pi Q(\lambda_{1z^\perp}))^{s/2+n/4}}
 F\left(\frac{s}{2}+\frac{n}{4},\, \frac{s}{2}-\frac{n}{4}+1, \, s+1; \, \frac{m_1}{Q(\lambda_{1z^\perp})} \right),
\]
see e.g.~\cite{B2} p.~215 (11). Consequently,
\begin{align*}
\Phi_{m,\mu}(z,h,s)&=\frac{\Gamma(\frac{s}{2}+\frac{n}{4})}{\Gamma(s+1)}\\
&\phantom{=}{}\times
\sum_{\substack{\lambda\in h(\mu+L)\\Q(\lambda)=m}}
\left(\frac{m_1}{Q(\lambda_{1z^\perp})}\right)^{\frac{s}{2}+\frac{n}{4}}
 F\left(\frac{s}{2}+\frac{n}{4},\, \frac{s}{2}-\frac{n}{4}+1, \, s+1; \, \frac{m_1}{Q(\lambda_{1z^\perp})} \right).
\end{align*}

We now prove that the sum converges for $(z,h)$ outside a subset  of measure zero to an integrable function on $X_K$.
By reduction theory, the arithmetic group $\Gamma_K=H(\Q)\cap K$ acts with finitely many orbits on the set of $\lambda\in \mu+L$ with $Q(\lambda)=m$.
We consider for a fixed $\lambda\in \mu+L$ with $Q(\lambda)=m$ the sum
\begin{align*}
S(z):=\sum_{\substack{\gamma\in \Gamma_{K,\lambda}\bs \Gamma_K}}
\phi(\gamma\lambda,z,s),
\end{align*}
where $\Gamma_{K,\lambda}=H_\lambda(\Q)\cap K$ and $H_\lambda$ denotes the stabilizer of $\lambda$ in $H$.
It suffices to show that $S(z)$
converges outside a subset of measure zero and that $\int_{\Gamma_K\bs \D} S(z)\,d\mu(z)<\infty$.
According to  Fubini's theorem, we have
\[
\int_{\Gamma_K\bs \D} S(z)\,d\mu(z)=\int_{\Gamma_{K,\lambda}\bs \D} \phi(\lambda,z,s)\,d\mu(z),
\]
and the desired convergence statement follows if we show that the integral on the right hand side is finite.
Fixing a base point $z_0\in \D$ with $z_0\perp \lambda$, we may realize $\D$ as the coset space of $H(\R)$ modulo the maximal compact subgroup given by the stabilizer of $z_0$.
Hence it suffices to show that
\[
\int_{\Gamma_{K,\lambda}\bs H(\R)} \phi(\lambda,hz_0,s)\,dh<\infty,
\]
where $dh$ denotes a Haar measure on $H(\R)$.
The latter integral is equal to
\begin{align*}
\int_{\Gamma_{K,\lambda}\bs H(\R)} \phi(h^{-1}\lambda,z_0,s)\,dh
&=\int_{h\in\Gamma_{K,\lambda}\bs H_\lambda(\R)} \int_{h'\in H_\lambda(\R)\bs H(\R)} \phi(h'{}^{-1}h^{-1}\lambda,z_0,s)\,dh'\,dh\\
&=\vol(\Gamma_{K,\lambda}\bs H_\lambda(\R)) \int_{H_\lambda(\R)\bs H(\R)} \phi(h'{}^{-1}\lambda,z_0,s)\,dh'.
\end{align*}
The convergence of the latter integral for $\Re(s)>s_0$ is proved in \cite{OT}, Proposition 3.1.1 (see also \cite{BK}, Section 4.2, for a comparison of the different setups).
\end{proof}

Next we consider the singularities of $\Phi_{m,\mu}(z,h,s)$.

\begin{lemma}
\label{lem:fin}
Let $m\in  F$ be totally positive. i) For all $\lambda\in L'$ with $Q(\lambda)=m$ we have
$0< m_1\leq Q(\lambda_{1z^\perp})$, and
\[
\frac{m_1}{Q(\lambda_{1z^\perp})}= \frac{2m_1}{m_1+(Q(\lambda_{1z^\perp})-Q(\lambda_{1z}))}.
\]
ii) For any $\eps>0$ and any compact subset $C\subset \D$, there are only finitely many $\lambda\in L'$ with $Q(\lambda)=m$ and
$\eps< m_1/Q(\lambda_{1z^\perp})$ for some $z\in C$.
\end{lemma}

\begin{proof}
This follows by a straightforward computation and the fact that $Q(\lambda_{1z^\perp})-Q(\lambda_{1z})$ is a positive definite quadratic form on $V_{\sigma_1}$.
\end{proof}

\begin{theorem}
\label{thm:lift3}
Let $\mu\in L'/L$, $m\in  \partial_F^{-1}+Q(\mu)$, and $m\gg 0$.
The series \eqref{eq:ser} and all its partial derivatives converge normally for $\Re(s)>s_0$ and $(z,h)\in X_K\setminus Z(m,\mu)$.
For any point $(z_0,h_0)\in \D\times H(\hat \Q)$ there is a neighborhood $U$ such that the function
\begin{align}
\label{eq:lift3}
\Phi_{m,\mu}(z,h,s)-
\sum_{\substack{\lambda\in h_0(\mu+L)\\Q(\lambda)=m\\ \lambda_1\perp z_0}}
\phi(\lambda,z,s)
\end{align}
is $C^\infty$ on $U$. Here the latter sum is finite.
\end{theorem}

\begin{proof}
Replacing the lattice $L$ by $h_0 L$, we may assume without loss of generality that $h_0=1$.
The condition $\lambda_1\perp z_0$ means that $\lambda_1$ is contained in the $n$-dimensional positive definite subspace $z_0^\perp$ of $V_{\sigma_1}$.
Since $V_{\sigma_2},\dots ,V_{\sigma_d}$ are positive definite,  there are only finitely many $\lambda$ in the coset $\mu+L$ satisfying the condition under the sum in \eqref{eq:lift3}.

Let $U'\subset \D$ be a compact neighborhood of $z_0$. Let $S_1$ be the set of all $\lambda\in \mu+L$
with $Q(\lambda)=m$ and $ m_1/Q(\lambda_{1z^\perp})<1/2$ for all $z\in U'$.
Let $S_2$ be the set of all $\lambda\in \mu+L$
with $Q(\lambda)=m$ and $ m_1/Q(\lambda_{1z^\perp})\geq 1/2$ for some $z\in U'$.
Then we have for $h\in K\subset H(\hat \Q)$ that
\[
\Phi_{m,\mu}(z,h,s)
=
\sum_{\substack{\lambda\in h(\mu+L)\\Q(\lambda)=m}}\phi(\lambda,z,s)
=\sum_{\lambda\in S_1}\phi(\lambda,z,s)+\sum_{\lambda\in S_2}\phi(\lambda,z,s).
\]
According to Lemma \ref{lem:fin}, the sum over $S_2$ is finite.
Moreover,
\[
\sum_{\lambda\in S_2}\phi(\lambda,z,s)-\sum_{\substack{\lambda\in \mu+L\\Q(\lambda)=m\\ \lambda_1\perp z_0}}
\phi(\lambda,z,s)
\]
is a smooth function in a small neighborhood $U$ of $(z_0,1)$.

Hence it suffices to show that the sum over $S_1$ converges normally for $z\in U'$.
Using the power series expansion of the Gauss hypergeometric function, we see that
\[
\phi(\lambda,z,s)\ll Q(\lambda_{1z^\perp})^{-s/2-n/4}
\]
for all $\lambda\in S_1$ and all $z\in U'$. The same bound (with different implied constants) holds for all iterated partial derivatives of $\phi(\lambda,z,s)$.
Consequently, it suffices to show that
\[
\sum_{\substack{\lambda\in h(\mu+L)\\Q(\lambda)=m}}Q(\lambda_{1z^\perp})^{-s/2-n/4}
\]
converges normally on $U$. This can be proved by comparing with an integral as in the proof of Theorem \ref{thm:lift2}.
\end{proof}


Since the Green function $\Phi_{m,\mu}(z,h,s)$ belongs to
$L^1(X_K)$, we may view it as a current $[\Phi_{m,\mu,s}]$, that is,
as a functional on smooth top degree differential forms on $X_K$ with
compact support.  For a bounded $C^\infty$-function $\alpha$ on $X_K$, we write
\[
[\Phi_{m,\mu,s}](\alpha) = \int_{X_K}
\Phi_{m,\mu}(z,h,s)\alpha(z)\, d\mu(z).
\]
Applying the Laplace operator, we obtain another current $\Delta_\D
[\Phi_{m,\mu,s}]$, given by
\[
(\Delta_\D[\Phi_{m,\mu,s}])(\alpha) = \int_{X_K} \Phi_{m,\mu}(z,h,s)(\Delta_\D\alpha)(z)\, d\mu(z).
\]
Moreover, for any divisor $Y$ on $X_K$, there is a Dirac current
$\delta_Y$, given by
\[
\delta_Y(\alpha)= \int_{Y} \alpha(z) \Omega^{n-1}.
\]
We define the degree of the divisor $Y$ by
$\deg(Y)=\delta_Y(1)$, provided the integral converges.

\begin{remark}
\label{rem:deg0}
The degree of the first Chern class in $\ch^1(X_K)$ of the line bundle of modular forms of weight $w$ is given by  $\deg c_1(\calM_w) =w \vol(X_K)$.
\end{remark}

\begin{proof}
The Petersson metric defines a hermitean metric on the line bundle of modular forms of weight $w$. Its first Chern form is $w\Omega$. The remark follows from the Poincar\'e--Lelong lemma.
\end{proof}

\begin{theorem}
\label{thm:diffeq}
Let $\Re(s)>s_0$.  The Green current $[\Phi_{m,\mu,s}]$ satisfies the following  differential equation
\[
\Delta_\D[\Phi_{m,\mu,s}] = \frac{1}{8}\left(s^2-s_0^2\right)[\Phi_{m,\mu,s}] - \frac{n}{4\Gamma(\frac{s}{2}-\frac{s_0}{2}+1)}\delta_{Z(m,\mu)}.
\]
\end{theorem}

\begin{proof}
In view of \eqref{eq:compot}, the result follows from \cite{OT}, Theorem 3.2.1 (3) and Corollary 3.2.1.
The comparison of the different normalizations of the Laplacian and the invariant measures is similar as in the proof of Theorem 4.7 of \cite{BK}.
\end{proof}

\subsection{The spectral expansion}

For $d=1$ (when $F=\Q$) a meromorphic continuation of the Green function $\Phi_{m,\mu}(z,h,s)$ is obtained in \cite{Br2},
employing Poincar\'e series built out of the functions $f_{m,\mu}(\tau,s)$ similarly as in Proposition \ref{prop:wm}.
Since the corresponding Poincar\'e series do not converge when $d>1$, we cannot argue this way. Instead, we use the approach of \cite{OT} \S6 and \cite{MW} \S4 to prove meromorphic continuation in the sense of distributions by means of spectral theory. Moreover, we refine the argument to obtain a meromorphic continuation as a smooth function on $X_K\setminus Z(m,\mu)$.

We begin by computing the spectral expansion of $\Phi_{m,\mu}(z,h,s)$.
Throughout this subsection we assume that $X_K$ is compact such that the Laplace operator has a discrete spectrum. This is always the case when $d>1$.

Recall that the Laplace operator  $-\Delta_\D$ gives rise to a densely defined self-adjoint operator on $L^2(X_K)$ which is positive.
Let $\Lambda\subset \R_{\geq 0}$ be the set of eigenvalues of $-\Delta_\D$. It is a countable set with no accumulation points. So we may write
$\Lambda=\{\lambda_k;\; k\in \Z_{\geq 0}\}$ with
\[
\lambda_0\leq \lambda_1\leq \lambda_2\leq \dots,
\]
where every $\lambda \in \R$ occurs with multiplicity $d(\lambda)$ given by the dimension of the corresponding eigenspace.
Let $\{\varphi_k\}\subset \C^\infty(X_K)$ be an orthonormal system of eigenfunctions such that $-\Delta_\D \varphi_k=\lambda_k\varphi_k$.
For any $k$ we choose $\alpha_k\in \C$ such that
\[
\lambda_k= -\frac{1}{8}\left( \alpha_k^2-s_0^2\right).
\]
Then $\alpha_k\in [-s_0,s_0]\cup i\R$.
For the eigenvalue $0$, the multiplicity $d(0)$ is equal to the number of connected components of $X_K$.
We choose the suitably normalized characteristic functions of the components of $X_K$ as an orthonormal basis of the eigenspace.

Any function $\varphi\in L^2(X_K)$ has a spectral decomposition
\begin{align}
\label{eq:spectral}
\varphi= \sum_{k=0}^\infty (\varphi,\varphi_k) \varphi_k,
\end{align}
where $(\varphi,\psi)=\int_{X_K} \varphi(z)\overline{\psi(z)}\,d\mu(z)$ denotes the scalar product on $L^2(X_K)$. The series converges in the $L^2$-norm.
We now compute this expansion for $\Phi_{m,\mu}(z,h,s)$.

\begin{theorem}
\label{thm:spectral}
Assume that $X_K$ is compact.
The Green function $\Phi_{m,\mu}(z,h,s)$ has the spectral expansion
\[
\Phi_{m,\mu}(z,h,s) = \frac{2n}{\Gamma(\frac{s}{2}-\frac{s_0}{2}+1)} \sum_{k=0}^\infty \frac{\delta_{Z(m,\mu)}(\overline{\varphi_k})}{s^2-\alpha_k^2}\cdot \varphi_k(z,h).
\]
\end{theorem}

\begin{proof}
This is an immediate consequence of Theorem \ref{thm:diffeq} and \eqref{eq:spectral}.
\end{proof}

\begin{corollary}
\label{cor:res}
As a distribution on $X_K$, the Green function $\Phi_{m,\mu}(z,h,s)$ has a meromorphic continuation in $s$ to the whole complex plane.
At $s=s_0$ it has a simple pole with residue
\begin{align}
\label{eq:A}
A(m,\mu):=2\frac{\deg(Z(m,\mu))}{\vol(X_K)}.
\end{align}
\end{corollary}

\begin{proof}
Using the Bessel inequality, one sees that the spectral expansion of $\Phi_{m,\mu}(z,h,s)$ given in Theorem \ref{thm:spectral} converges locally uniformly for
$s\in\C$ with $s\neq \pm\alpha_k$ to an element of $L^2(X_K)$.
This proves the meromorphic continuation.

The singularity at $s=s_0$ comes from the terms in the spectral expansion with eigenvalue $\lambda_k=0$. The sum over the remaining terms is holomorphic at $s_0$. Hence, the singularity is given by
\begin{align}
\label{res1}
\frac{2n}{\Gamma(\frac{s}{2}-\frac{s_0}{2}+1)}\sum_{\substack{k\geq 0\\ \lambda_k=0}}\frac{\delta_{Z(m,\mu)}(\overline{\varphi_k})}{s^2-s_0^2}\cdot \varphi_k.
\end{align}
The eigenfunctions $\varphi_k$ contributing to this sum are the characteristic functions of the components of $X_K$ multiplied by the normalizing factor
 $(d(0)/\vol(X_K))^{1/2}$.
It is easily seen that for such eigenfunctions $\delta_{Z(m,\mu)}(\overline{\varphi_k})=(d(0)\vol(X_K))^{-1/2}\delta_{Z(m,\mu)}(1)$.
Hence the sum \eqref{res1} is equal to
\[
\frac{2n(s^2-s_0^2)^{-1}}{\Gamma(\frac{s}{2}-\frac{s_0}{2}+1)}\cdot\frac{\deg(Z(m,\mu))}{\vol(X_K)}.
\]
This function has a simple pole at $s=s_0$ with the claimed residue.
\end{proof}

\begin{remark}
The spectral expansion also implies that $\Gamma(\frac{s}{2}-\frac{s_0}{2}+1)\Phi_{m,\mu}(z,h,s)$ is invariant under the substitution $s\mapsto -s$.
\end{remark}

We now refine the above argument, to obtain a meromorphic continuation in $s$ of  $\Phi_{m,\mu}(z,h,s)$ as a continuous  {\em function} on $X_K\setminus Z(m,\mu)$. Then the distribution differential equation of Theorem~\ref{thm:diffeq} and the elliptic regularity theorem imply that $\Phi_{m,\mu}(z,h,s)$ is actually real analytic on $X_K\setminus Z(m,\mu)$. The following lemma is known, see e.g. \cite{Shubin}, Proposition 10.2. We include it here for completeness.

\begin{lemma}
\label{lem:uniconv}
Assume the notation of the beginning of this subsection.
\begin{enumerate}
\item If $t>n$, then the series
$\sum_{k\geq 0} (\lambda_k+1)^{-t}$
converges.
\item For any integer $N>n$, there is a constant $C>0$ such that for all $k\geq 0$ we have
\[
\|\varphi_k\|_\infty \leq C (\lambda_k+1)^{N}.
\]
\item
If $\psi\in C^\infty(X_K)$, then for any $N\in \Z_{\geq 0}$
and for all $k\in \Z_{\geq 0}$ we have
\[
|(\psi,\varphi_k)|\leq  (\lambda_k+1)^{-N} \|(-\Delta_\D+1)^N \psi\|_2.
\]
\item
If $\psi\in C^\infty(X_K)$, then the spectral expansion \eqref{eq:spectral} converges uniformly towards $\psi$.
\end{enumerate}
\end{lemma}

\begin{proof}
The first assertion is a consequence of Weyl's law which states that
\[
\#\{k;\; \lambda_k\leq x\}\sim cx^{\delta/2}, \quad x\to \infty,
\]
where $c>0$ is a constant and
$\delta$ is the dimension (over $\R$) of the compact real Riemann manifold $X_K$.

The second assertion follows from the fact that for any integer $N>\delta/2$ the pseudo differential operator $(-\Delta_\D+1)^{-N}$ has a continuous kernel function in $C(X_K\times X_K)$.

The third statement is an easy consequence of the self-adjointness of $-\Delta_\D$ and the Cauchy-Schwartz inequality. Finally, the last statement follows from (1), (2) and (3).
\end{proof}

\begin{theorem}
\label{thm:smoothcont}
For $(z,h)\in X_K\setminus Z(m,\mu)$, the Green function $\Phi_{m,\mu}(z,h,s)$ has a meromorphic continuation in $s$ to the whole complex plane. For fixed $s$ outside the set of poles, the resulting function in $(z,h)$ is real analytic  on $X_K\setminus Z(m,\mu)$.
\end{theorem}

\begin{proof}
Let $\sigma:[0,1]\to \R$ be a monotonous $C^\infty$-function such that $\sigma(t)=1$ for $t\leq 1/2$ and $\sigma(t)=0$ for $t\geq 3/4$.
Besides the Poincar\'e series $\Phi_{m,\mu}(z,h,s)$ (see Theorem \ref{thm:lift2}), we consider the Poincar\'e series
\begin{align*}
\tilde\Phi_{m,\mu}(z,h,s)
&=
\sum_{\substack{\lambda\in h(\mu+L)\\Q(\lambda)=m}}
\sigma\left(\frac{m_1}{Q(\lambda_{1z^\perp})} \right)\phi(\lambda,z,s),\\
F(z,h,s)&=
\sum_{\substack{\lambda\in h(\mu+L)\\Q(\lambda)=m}}
\left[\Delta_\D\left(\sigma\left(\frac{m_1}{Q(\lambda_{1z^\perp})}\right)\phi(\lambda,z,s)\right)-\sigma\left(\frac{m_1}{Q(\lambda_{1z^\perp})}\right)\Delta_\D\left(\phi(\lambda,z,s)\right)\right].
\end{align*}

The difference of $\Phi_{m,\mu}(z,h,s)$ and $\tilde\Phi_{m,\mu}(z,h,s)$ is the Poincar\'e series
\[
\sum_{\substack{\lambda\in h(\mu+L)\\Q(\lambda)=m}}
\left(1-\sigma\left(\frac{m_1}{Q(\lambda_{1z^\perp})} \right)\right)\phi(\lambda,z,s).
\]
As in the proof of Theorem \ref{thm:lift3} we see that it is locally finite, that is, for $(z,h)$ in any compact subset of $\D\times H(\hat \Q)$, only finitely many terms are non-zero. Hence it defines a holomorphic function for all $s\in \C$, which is smooth for $(z,h)\in X_K\setminus Z(m,\mu)$.
We now show that $\tilde\Phi_{m,\mu}(z,h,s)$ has a meromorphic continuation in $s$ which is continuous for  $(z,h)\in X_K$. This implies the desired continuation of $\Phi_{m,\mu}(z,h,s)$.

As in the proof of Theorem \ref{thm:lift3} we see that $\tilde\Phi_{m,\mu}(z,h,s)$ converges normally for $\Re(s)>s_0$ and defines a smooth function on $X_K$.
The Poincar\'e series $F(z,h,s)$ is locally finite and defines a smooth function, which is holomorphic in $s$ on the whole complex plane.
Moreover, the differential equation \eqref{eq:diffphi}
implies that
\begin{align}
\label{deq}
\Delta_\D  \tilde\Phi_{m,\mu}(z,h,s)= \frac{1}{8}\left(s^2-s_0^2\right)\tilde\Phi_{m,\mu}(z,h,s) + F(z,h,s)
\end{align}
for $\Re(s)>s_0$.
Hence the coefficients of the spectral expansion of $\tilde \Phi_{m,\mu}(z,h,s)$ are given by
\[
(\tilde\Phi_{m,\mu}(\cdot,s), \varphi_k)= \frac{8}{\alpha_k^2-s^2} (F(\cdot,s),\varphi_k),
\]
and we have
\[
\tilde\Phi_{m,\mu}(z,h,s) = \sum_{k= 0}^\infty  \frac{8}{\alpha_k^2-s^2} (F(\cdot,s),\varphi_k) \varphi_k(z,h).
\]
Lemma  \ref{lem:uniconv} implies that the series converges locally uniformly for $s\in \C$ and $(z,h)\in X_K$. Consequently, it defines a meromorphic continuation in $s$ which is continuous in $(z,h)$.

Now the distribution differential equation of Theorem~\ref{thm:diffeq} and the elliptic regularity theorem imply that for fixed $s$ outside the set of poles, $\Phi_{m,\mu}(z,h,s)$ is actually real analytic for $(z,h)\in X_K\setminus Z(m,\mu)$.
\end{proof}

\subsection{Regularized Green functions}

Here we define the regularized theta lift of a {\em harmonic} Whittaker form, that is, a Whittaker form with parameter $s_0$.
We determine the singularities of the lift.
In this subsection we do not  have to assume that $X_K$ is compact. So we come back to the general setup of Section \ref{sect2}.

\begin{definition}
Let $f\in  H_{k,\bar\rho_L}$ and write
\begin{align}
\label{eq:fdecomp}
f=\sum_{\mu\in L'/L} \sum_{m\gg 0} c(m,\mu)f_{m,\mu}(\tau).
\end{align}
We define the regularized theta lift $\Phi(z,h,f)$ of $f$ to be the constant term in the Laurent expansion at $s=s_0$ of
\[
\Phi(z,h,s,f):=\sum_{\mu\in L'/L} \sum_{m\gg 0} c(m,\mu)\Phi_{m,\mu}(z,h,s).
\]
\end{definition}

If $\mu\in L'/L$ and $m\in  \partial_F^{-1}+Q(\mu)$ is totally positive, we briefly write $\Phi_{m,\mu}(z,h)$ for the regularized theta lift of the harmonic Whittaker form $f_{m,\mu}(\tau)$, that is, for the constant term in the Laurent expansion of $\Phi_{m,\mu}(z,h,s)$ at $s=s_0$.

For a harmonic Whittaker form $f\in  H_{k,\bar\rho_L}$ as in \eqref{eq:fdecomp} we define a divisor $Z(f)\in \Div(X_K)_\C$ by
\begin{align}
\label{eq:divf}
Z(f)=\sum_{\mu\in L'/L} \sum_{m\gg 0} c(m,\mu)Z(m,\mu).
\end{align}
Moreover, by means of the quantities $A(m,\mu)$ of Corollary \ref{cor:res}
we define
\begin{align}
\label{eq:Af}
A(f)=\sum_{\mu\in L'/L} \sum_{m\gg 0} c(m,\mu)A(m,\mu).
\end{align}
In view of Corollary \ref{cor:res} we have
\begin{align}
\Phi(z,h,f)= \lim_{s\to s_0} \left(\Phi(z,h,s,f)-\frac{ A(f)}{s-s_0} \right).
\end{align}

If $Y$ is an irreducible Cartier divisor on a normal complex space $X$, we say that a real analytic function $F$ on $X\setminus Y$ has a logarithmic singularity along $Y$, if for any $x\in Y$ there is a neighborhood $U\subset X$ and a local equation $G$ for $Y$ such that $F-\log|G|$ can be continued to a real analytic function on $U$. We extend this definition $\C$-linearly to $\Div(X_K)_\C$.

\begin{theorem}
\label{thm:reglift}
Assume that $n>0$. For $f\in  H_{k,\bar\rho_L}$
the function $\Phi(z,h,f)$ is real analytic on $X_K\setminus Z(f)$. It has a logarithmic singularity along the divisor $-2Z(f)$.
\end{theorem}

\begin{proof}
It suffices to show that  for $\mu\in L'/L$ and totally positive  $m\in  \partial_F^{-1}+Q(\mu)$, the function $\Phi_{m,\mu}(z,h)$ is real analytic on $X_K\setminus Z(m,\mu)$ and has a logarithmic singularity along the divisor $-2Z(m,\mu)$. To this end we show that for any point $(z_0,h_0)\in \D\times H(\hat \Q)$ the function
\begin{align}
\label{eq:reglift}
\Phi_{m,\mu}(z,h)+
\sum_{\substack{\lambda\in h_0(\mu+L)\\Q(\lambda)=m\\ \lambda_1\perp z_0}}
\log |Q(\lambda_{1z})|
\end{align}
is real analytic in a neighborhood of $(z_0,h_0)$.

Since the residue of $\Phi_{m,\mu}(z,h,s)$ at $s=s_0$ does not depend
on $(z,h)$, the proof of Theorem \ref{thm:smoothcont} also shows that
the function
\begin{align*}
\Phi_{m,\mu}(z,h)-
\sum_{\substack{\lambda\in h_0(\mu+L)\\Q(\lambda)=m\\ \lambda_1\perp z_0}}
\phi(\lambda,z,s_0)
\end{align*}
is real analytic in a neighborhood of $(z_0,h_0)$.
Hence it suffices to show that
\[
\phi(\lambda,z,s_0)+\log |Q(\lambda_{1z})|
\]
extends to a real analytic function on $\D$.
This follows from Lemma \ref{lem:reglift} below.
%
\end{proof}

\begin{lemma}
\label{lem:reglift}
If $n>0$, the function
\[
\frac{2}{n}w^{n/2}F(n/2,1,n/2+1,w) + \log(1-w)
\]
extends to a real analytic function near w=1.
\end{lemma}

\begin{proof}
We use the integral representation
\begin{align*}
\frac{2}{n}w^{n/2}F(n/2,1,n/2+1,w)=\int_0^1(tw)^{n/2}(1-tw)^{-1}\frac{dt}{t}= \int_0^w t^{n/2}(1-t)^{-1}\frac{dt}{t},
\end{align*}
see for instance \cite{AS} (15.3.1). Comparing this with $\log(1-w)=-\int_0^w \frac{dt}{1-t}$, we see that
\begin{align*}
\frac{2}{n}w^{n/2}F(n/2,1,n/2+1,w)+ \log(1-w)= \int_0^w \frac{t^{n/2-1}-1}{1-t}\, dt.
\end{align*}
If $n=1$ this is equal to  $\int_0^w \frac{dt}{t+\sqrt{t}}$, if $n=2$ it vanishes identically, and if $n\geq 3$ it is equal to
\begin{align*}
-\sum_{k=0}^{n-3}\int_0^w \frac{t^{k/2}}{1+\sqrt{t}}\, dt.
\end{align*}
In all cases, the resulting function is real analytic near $w=1$.
\end{proof}

\begin{corollary}
\label{cor:pl}
The differential form $dd^c \Phi(f)$ extends to a smooth form on $X_K$, which is a (harmonic) Poincar\'e dual form for $Z(f)$.
The current $[\Phi(f)]$ induced by $\Phi(z,h,f)$ satisfies the $dd^c$-equation
\[
dd^c[\Phi(f)] +\delta_{Z(f)} = [dd^c \Phi(f)] .
\]
\end{corollary}

\begin{proof}
The corollary follows from Theorem \ref{thm:reglift} by means of the usual Poincar\'e-Lelong argument, see e.g.~\cite{SABK}, Chapter II.1.4, Theorem 2.
\end{proof}

\begin{remark}
\label{thm:regdiffeq}
The current $[\Phi(f)]$ induced by $\Phi(z,h,f)$ satisfies the differential equation
\[
\Delta_\D[\Phi(f)] + \frac{n}{4}\delta_{Z(f)}= \frac{n}{8}[A(f)] .
\]
\end{remark}

\begin{proof}
This is a direct consequence of Theorem \ref{thm:diffeq} and Corollary \ref{cor:res}.
\end{proof}


\section{The theta lift and meromorphic modular forms}
\label{sect:6}

We continue to use the notation of the previous section.
Here we investigate the relationship of the regularized theta lift and the Kudla--Millson lift (see \cite{KM1}, \cite{KM2}, \cite{KM3}).
We use the approach of \cite{BF}. As an application we construct explicit meromorphic modular forms on $X_K$ whose divisors are supported on Heegner divisors.
They are analogous to the automorphic products constructed by Borcherds
\cite{Bo2}. However, notice that there are no Fourier expansions and therefore no product expansions when $X_K$ is compact. Consequently, Borcherds' argument to prove important properties of the lift (such as e.g.~meromorphicity) cannot be employed when $d>1$.

\subsection{The relationship with a regularized Kudla--Millson lift}
Recall that $\kappa$ is the dual weight for $k$ given by
\[
\kappa=(2-k_1,k_2,\dots,k_d)=\left( \frac{n+2}{2},\dots,\frac{n+2}{2}\right).
\]

\begin{proposition}
\label{prop:ddcex}
Let $\mu\in L'/L$ and $m\in \partial_F^{-1}+Q(\mu)$ be totally positive. For $\Re(s)>s_0+2$ we have the identity
\[
dd^c\Phi_{m,\mu}(z,h,s)=\frac{1}{\sqrt{D}}
\int_{\tilde \Gamma_\infty\bs \H^d}^{reg}
\left\langle \overline{\delta_k(f_{m,\mu}(\tau,s))}, \Theta_{KM}(\tau, z,h)\right\rangle v^\kappa \, d\mu(\tau).
\]
Here $\delta_k$ is the differential operator defined in \eqref{defdelta}.
\end{proposition}

\begin{proof}
According to Proposition \ref{prop:BF} we have
\[
dd^c\Theta_{S}(\tau, z,h) = -L_\kappa^{(1)}\Theta_{KM}(\tau, z,h).
\]
Moreover, writing $\eta= (v_2\cdots v_d)^{\ell/2}d\tau_1 d\mu(\tau_2)\cdots d\mu(\tau_d)$,
we have the identity of differential forms on $\H^d$:
\[
 -\left(L_\kappa^{(1)}
\Theta_{KM}(\tau, z,h) \right) (v_2\cdots v_d)^{\ell/2}\, d\mu(\tau)
=\bar \partial (\Theta_{KM}(\tau, z,h)\eta ).
\]
In view of \eqref{Masy1}, when $\Re(s)$ is sufficiently large, we may interchange the regularized theta integral in the definition of $\Phi_{m,\mu}(z,h,s)$ with the operator $dd^c$. By means of the above identities we obtain
\begin{align*}
dd^c\Phi_{m,\mu}(z,h,s)&=
\frac{1}{\sqrt{D}}\int_{\tilde \Gamma_\infty\bs \H^d}^{reg}
\langle f_{m,\mu}(\tau,s), dd^c \Theta_S(\tau, z,h)\rangle
(v_2\cdots v_d)^{\ell/2}\,d\mu(\tau) \\
&=
\frac{1}{\sqrt{D}}\int_{\tilde \Gamma_\infty\bs \H^d}^{reg}
\langle f_{m,\mu}(\tau,s), -L_\kappa^{(1)} \Theta_{KM}(\tau, z,h)\rangle
(v_2\cdots v_d)^{\ell/2}\,d\mu(\tau)\\
&=\frac{1}{\sqrt{D}}\int_{\tilde \Gamma_\infty\bs \H^d}^{reg}
\langle f_{m,\mu}(\tau,s), \bar \partial \Theta_{KM}(\tau, z,h)\eta \rangle .
\end{align*}
Using the product rule, we find
\begin{align}
\label{eq:pr}
dd^c\Phi_{m,\mu}(z,h,s)&= \frac{1}{\sqrt{D}}\int_{\tilde \Gamma_\infty\bs \H^d}^{reg}
d\langle f_{m,\mu}(\tau,s), \Theta_{KM}(\tau, z,h)\eta \rangle\\
\nonumber
&\phantom{=}{}  - \frac{1}{\sqrt{D}}\int_{\tilde \Gamma_\infty\bs \H^d}^{reg}
\langle  \bar \partial(f_{m,\mu}(\tau,s)), \Theta_{KM}(\tau, z,h)\eta \rangle.
\end{align}
For the second summand on the right hand side  we notice that
\begin{align}
 \label{eq:pr2}
 \bar \partial(f_{m,\mu}(\tau,s))\eta &= -(L_k^{(1)} f_{m,\mu}(\tau,s)) (v_2\cdots v_d)^{\ell/2}\, d\mu(\tau)\\
\nonumber
& = -\overline{\delta_k(f_{m,\mu}(\tau,s))}(v_1\cdots v_d)^{\ell/2}\, d\mu(\tau).
\end{align}
Hence this term gives the right hand side of the formula stated in the proposition.

Consequently, it suffices to prove that the first summand on the right hand side of \eqref{eq:pr} vanishes for $\Re(s)>s_0+2$.
For $T>0$ we let $R_T\subset \R_{>0}^d$ be the rectangle
\[
R_T= [1/T,T]\times\dots \times [1/T,T].
\]
Using the invariance of the integrand under translations, we find by  Stokes' theorem that
\begin{align*}
\label{eq:term0}
\int_{\tilde \Gamma_\infty\bs \H^d}^{reg}
d\langle f_{m,\mu}(\tau,s), \Theta_{KM}(\tau, z,h)\eta \rangle
=\lim_{T\to \infty} \int_{\partial R_T}\int_{\calO_F\bs \R^d}
\langle f_{m,\mu}(\tau,s), \Theta_{KM}(\tau, z,h)\eta \rangle.
\end{align*}
Inserting \eqref{eq:kmex}, \eqref{thetakm2}, and \eqref{eq:deffms}, and carrying out the integration over $u$, we see that this is equal to
\begin{align*}
\sqrt{D}&C(m,k,s)
\lim_{T\to \infty} \int_{\partial R_T}\calM_s(-4\pi m_1 v_1)\\
&{}\times\sum_{\substack{\lambda\in h(\mu+L)\\Q(\lambda)=m}}
P_{KM}(z,\sqrt{v_1}\lambda_1)e^{-4\pi Q(\lambda_{1z^\perp})v_1} (v_2\cdots v_d)^{\ell/2-2}\,dv_2\cdots dv_d.
\end{align*}
Only the parts of the boundary where $v_1=1/T$ or $v_1=T$ give a non-zero contribution. Carrying out the integration over $v_2,\dots, v_d$, we see that the $v_1=1/T$ contribution is equal to a constant times
\begin{align*}
\lim_{T\to \infty} T^{k_1/2} M_{-k_1/2,\,s/2}(4\pi m_1/T)
\sum_{\substack{\lambda\in h(\mu+L)\\Q(\lambda)=m}}
P_{KM}(z,T^{-1/2}\lambda_1)
e^{-2\pi (Q(\lambda_{1z^\perp})- Q(\lambda_{1z}))/T}.
\end{align*}
The later sum is up to a constant factor equal to the $m$-th Fourier coefficient of
\[
\Theta_{KM}(u+i(1/T,1,\dots,1),z,h).
\]
Hence it converges and satisfies the growth estimate of Proposition \ref{prop:thetakmgrowth} as $1/T\to 0$.
So the asymptotic behavior of the $M$-Whittaker function \eqref{Masy1} implies that the limit vanishes for $\Re(s)>s_0+2$.
On the other hand, the $v_1=T$ contribution is easily seen to vanish for $(z,h)\in X_K\bs Z(m,\mu)$.
This proves the proposition.
%
\end{proof}

\subsection{Eisenstein series and theta integrals}

The right hand side of the formula of Proposition \ref{prop:ddcex} converges for $\Re(s)>s_0+2$. It has a meromorphic continuation to the whole complex plane, since the left hand side has. It is actually holomorphic at $s=s_0$.
We now modify the integral representation on the right hand side in order to obtain an expression which converges near $s_0$.
This is done by subtracting the ``Eisenstein contribution'' of  $\Theta_{KM}(\tau, z,h)$. The remaining ``cuspidal contribution'' satisfies a better growth estimate as $v_i\to 0$ and therefore leads to a larger domain of convergence.

We briefly summarize some facts on the Siegel--Weil formula, see e.g. \cite{KR1}, \cite{Ku:Integrals} for more details.
Let $\chi_V$ denote the quadratic character of $\A_F^\times/F^\times$ associated
to $V$ given by
\[
\chi_V(x)=(x,(-1)^{\ell(\ell-1)/2} \det(V))_F.
\]
Here $\det(V)$ denotes the Gram determinant of $V$ and
$(\cdot,\cdot)_F$ is the Hilbert symbol of $F$. Let $P\subset G$ be the parabolic subgroup of upper triangular matrices. For $s\in \C$ and a standard section $\Phi(s)$ of the  principal series representation $I(s,\chi_V)$ induced by
$\chi_V |\cdot|^s$, we have the Eisenstein series
\[
E(g,s,\Phi)= \sum_{\gamma\in P(F)\bs G(F)} \Phi(\gamma g).
\]
It converges for $\Re(s)>1$ and has a meromorphic continuation to the whole complex plane.

Recall that if $v=\sigma_j$ is an infinite prime, then the corresponding local induced representation $I(s,\chi_{V,\sigma_j})$ is generated by the sections
\[
\Phi_\R^{l_j}(k_\alpha,\phi)= \chi_{1/2}(k_\alpha,\phi)^{2l_j}= \pm e^{i l_j \alpha}
\]
for $l_j\in \frac{1}{2}\Z$ satisfying $l_j\equiv \ell/2\pmod{\Z}$. Here $(k_\alpha,\phi)\in \widetilde{\SO}_2(\R)$ is given by \eqref{eq:kalpha} and $\chi_{1/2}$ is the character defined in \eqref{eq:chi12}.
If $l=(l_1,\dots,l_d)$ is a $d$-tuple of such half-integers we put $\Phi_\infty^{l}=\prod_j \Phi_\R^{l_j}$.
If $\Phi_f(s)$ is a standard section of the non-archimedian induced representation,
we obtain an
Eisenstein series of weight $l$ on $\H^d$ by putting
\[
E(\tau,s,l;\Phi_f) = v^{-l/2} E(\tilde g_\tau,s,\Phi_f\otimes \Phi_\infty^{l}),
\]
where $\tilde g_\tau\in \tilde G_\R$ with the property that $\tilde g_\tau (i,\dots,i) = \tau$.


The Weil representation gives rise to a $\tilde G_\A$-intertwining map
\begin{equation}
\lambda: S(V(\A_F))\longrightarrow I(s_0,\chi_V),\quad
\lambda(\varphi)(g)=(\omega(g)\varphi)(0),
\end{equation}
where $s_0=\ell/2-1=n/2$. We also write $\lambda(\varphi)$ for the unique standard section of $I(s,\chi_V)$ whose value at $s_0$ is equal to $\lambda(\varphi)$.
The map $\lambda$ factors into $\lambda=\lambda_\infty\otimes \lambda_f$, where $\lambda_\infty$  and $\lambda_f$ are the analogous intertwining maps at the finite and the infinite places, respectively.
We obtain a vector valued Eisenstein series for $\tilde\Gamma$ of weight $l$ with representation $\rho_L$ by putting
\begin{align}
E_L(\tau,s,l)=\sum_{\mu\in L'/L} E(\tau,s,l;\lambda_f( \chi_\mu))\chi_\mu.
\end{align}
Note that if the class number of $F$ is one, we have that
\begin{align*}
E_L(\tau,s,l) = \sum_{\gamma\in \tilde\Gamma_\infty\bs \tilde \Gamma} \left(v_1^{(s+1-l_1)/2}\cdots v_d^{(s+1-l_d)/2}\chi_0\right)\mid_{l,\rho_L} \gamma.
\end{align*}
In general, it is a finite sum over such Eisenstein series.
We will be interested in the special value $E_L(\tau,\kappa):=E_L(\tau,s_0,\kappa)$ at $s_0$.

For the rest of this section we
assume that $V$ is anisotropic over $F$ or that its Witt rank
is smaller than $n$.
Note that this condition is automatically fulfilled when $d>1$ or $n>2$.
Employing the Siegel--Weil formula (see e.g.~\cite{KR1} and \cite{We2}), it can be shown that
the average value of the Kudla--Millson theta function on $X_K$ is the given by an Eisenstein series of weight $\kappa$. More precisely, we have
\begin{align}
\label{eq:kms}
E_L(\tau,\kappa)= -\frac{1}{\vol(X_K)}\int_{X_K} \Theta_{KM}(\tau,z,h) \Omega^{n-1}.
\end{align}
Kudla proved this for $F=\Q$ in \cite[Corollary 4.16]{Ku:Integrals}, and the argument for general $F$ is analogous.
Moreover, as in \cite[Remark 2.8]{Ku:Integrals}, it can be proved that $E_L(\tau,\kappa)$ is holomorphic in $\tau$ and therefore defines an element of  $M_{\kappa,\rho_L}$.

At the cusp $\infty$, the Eisenstein series has a Fourier expansion of the form
\begin{align}
E_L(\tau,\kappa)= \chi_0 + \sum_{\mu\in L'/L}\sum_{m>0} B(m,\mu)e(\tr(m\tau))\chi_\mu.
\end{align}
The coefficients $B(m,\mu)$ can be computed explicitly using the argument of
\cite{KY}, \cite{Scho}, or \cite{BK}. However, we will not need that.

The differential form $-E_L(\tau,\kappa)\Omega$ can be viewed as the average of  $\Theta_{KM}(\tau,z,h) $.
We define the cuspidal part of the Kudla--Millson theta function by
\begin{align}
\tilde \Theta_{KM}(\tau,z,h)=\Theta_{KM}(\tau,z,h)+E_L(\tau,\kappa)\Omega.
\end{align}
It is rapidly decreasing at all cusps of $\tilde\Gamma$.

\begin{proposition}
\label{prop:thetatilde}
Assume the above hypothesis on  $V$.
\begin{enumerate}
\item[(i)]
The function  $\tilde \Theta_{KM}(\tau,z,h)v^{\kappa/2}$ is  bounded  on $\H^d$.
\item[(ii)]
For $v_i\to 0$ we have uniformly in $u$ that $\tilde \Theta_{KM}(\tau,z,h)=O(v^{-\kappa/2})$.
\item[(iii)]
For  $m\in \partial_{F}^{-1}$ the $m$-th Fourier coefficient of $\tilde \Theta_{KM}(\tau,z,h)$ is bounded by
$O(v_1^{-\kappa_1/2})$ as $v_i\to 0$.
\end{enumerate}
\end{proposition}

\begin{proof}
Since  $\tilde \Theta_{KM}(\tau,z,h)$ is rapidly decreasing,
(i) and (ii) follow by the usual argument. It remains to prove (iii). The behavior of the Fourier coefficients as $v_1\to 0$ is a direct consequence of (ii). Moreover, since $\tilde \Theta_{KM}(\tau,z,h)$ is holomorphic in $\tau_2,\dots,\tau_d$ its Fourier coefficients are bounded as $v_i\to 0$ for $i=2,\dots,d$.
\end{proof}

\begin{proposition}
\label{prop:ddcex2}
We have
\[
dd^c\Phi_{m,\mu}(z,h,s)=\frac{1}{\sqrt{D}}
\int_{\tilde \Gamma_\infty\bs \H^d}^{reg}
\left\langle \overline{\delta_k(f_{m,\mu}(\tau,s))}, \tilde \Theta_{KM}(\tau,z,h)\right\rangle v^{\kappa}\, d\mu(\tau)-\frac{B(m,\mu)\Omega}{\Gamma(\frac{s}{2}-\frac{s_0}{2}+1)}.
\]
Here the regularized integral converges locally uniformly for $\Re(s)>1$.
\end{proposition}

\begin{proof}
According to Proposition \ref{prop:ddcex} we have
\begin{align}
\label{eq:for1}
dd^c\Phi_{m,\mu}(z,h,s)&=\frac{1}{\sqrt{D}}
\int_{\tilde \Gamma_\infty\bs \H^d}^{reg}
\left\langle \overline{\delta_k(f_{m,\mu}(\tau,s))}, \tilde\Theta_{KM}(\tau, z,h)\right\rangle v^{\kappa}\, d\mu(\tau)\\
\nonumber
&\phantom{=}{}
-\frac{1}{\sqrt{D}}
\int_{\tilde \Gamma_\infty\bs \H^d}^{reg}
\left\langle \overline{\delta_k(f_{m,\mu}(\tau,s))}, E_L(\tau,\kappa)\Omega\right\rangle v^{\kappa}\, d\mu(\tau).
\end{align}
We have to compute the latter integral. A  direct computation shows that
\begin{align}
\label{eq:for2}
\overline{\delta_k(f_{m,\mu}(\tau,s))}&=C(m,k,s)(s+s_0)(4\pi m_1)^{-k_1/2} \\
\nonumber
&\phantom{=}{}\times v_1^{k_1/2-1}M_{1-k_1/2,s/2}(4\pi m_1 v_1)e^{-2\pi m_2 v_2}\cdots e^{-2\pi m_d v_d}e(-\tr(mu))\chi_\mu.
\end{align}
Inserting this and carrying out the integrations over $u$ and $v_2,\dots ,v_d$, we see that the second integral on the right hand side of \eqref{eq:for1}
is equal to
\begin{align*}
B(m,\mu)\Omega \cdot\frac{s+s_0}{\Gamma(s+1)} \int_0^\infty
(4\pi m_1 v_1)^{-k_1/2}M_{1-k_1/2,s/2}(4\pi m_1 v_1)e^{-2\pi m_1 v_1} \frac{dv_1}{v_1}.
\end{align*}
This is a Laplace transform, which can be computed  by means of \cite{B2} p.~215 (11). We obtain for the second integral on the right hand side of \eqref{eq:for1}
\begin{align*}
\frac{B(m,\mu)\Omega}{\Gamma(\frac{s}{2}-\frac{s_0}{2}+1)}.
\end{align*}
This proves the formula of the proposition.

We now prove the convergence statement for the integral.
According to \eqref{eq:for2} and \eqref{Masy1} we have
\[
\overline{\delta_k(f_{m,\mu}(\tau,s))}=O(v_1^{\Re(s)/2+(k_1-1)/2}),\quad v_i\to 0.
\]
By means of Proposition \ref{prop:thetatilde} (iii), we see that
\[
\int_{\calO_F\bs \R^d}
\left\langle \overline{\delta_k(f_{m,\mu}(\tau,s))}, \tilde \Theta_{KM}(\tau,z,h)\right\rangle v^{\kappa}\,du
=O\left(v_1^{\Re(s)/2+1/2}(v_2\cdots v_d)^{n/2+1}\right),
\]
as $v_i\to 0$. On the other hand, in view of \eqref{Masy2}, this quantity is bounded as $v_i\to \infty$.
Consequently,
\[
\int_{v\in \R_{>0}^d}\left(\int_{\calO_F\bs \R^d}
\left\langle \overline{\delta_k(f_{m,\mu}(\tau,s))}, \tilde \Theta_{KM}(\tau,z,h)\right\rangle v^\kappa\,du\right)  \frac{dv}{\norm(v)^{2} }
\]
converges when $\Re(s)>1$.
\end{proof}

\subsection{Regularized Green functions and the Kudla--Millson lift of cusp forms}

For a cusp form $g\in S_{\kappa,\rho_L}$ we define the Kudla--Millson lift by
\begin{align}
\Lambda(z,h,g) = \big( \Theta_{KM}(\tau,z,h), g(\tau)\big)_{Pet}.
\end{align}
The theta integral converges and defines a closed harmonic $2$-form on $X_K$. The following theorem is a generalization of
\cite{BF} Theorem 6.1 to our situation.

\begin{theorem}
\label{thm:bokm}
Let $f\in  H_{k,\bar\rho_L}$. We write
\begin{align*}
f=\sum_{\mu\in L'/L} \sum_{m\gg 0} c(m,\mu)f_{m,\mu}(\tau),
\end{align*}
and define
\begin{align}
\label{eq:defbf}
B(f)=\sum_{\mu\in L'/L} \sum_{m\gg 0} c(m,\mu)B(m,\mu).
\end{align}
Then we have the identity
\begin{align*}
dd^c\Phi(z,h,f)=\Lambda(z,h,\xi_k(f))-B(f)\Omega.
\end{align*}
\end{theorem}

\begin{proof}
We first assume that $n>2$. Then it follows from Proposition~\ref{prop:ddcex2}
that the regularized Green function $\Phi_{m,\mu}(z,h)$ satisfies the identity
\begin{align}
\label{eq:bokme}
dd^c\Phi_{m,\mu}(z,h)=\frac{1}{\sqrt{D}}
\int_{\tilde \Gamma_\infty\bs \H^d}
\left\langle \overline{\delta_k(f_{m,\mu}(\tau))}, \tilde \Theta_{KM}(\tau,z,h)\right\rangle v^{\kappa}\, d\mu(\tau)-B(m,\mu)\Omega.
\end{align}
Notice that the regularized theta integral in Proposition~\ref{prop:ddcex2} converges near $s=s_0$ when $n>2$.
Moreover, because of  \eqref{eq:delta} and Proposition \ref{prop:thetatilde}, the theta integral in \eqref{eq:bokme} is well defined without any
regularization!
By the unfolding argument we see that it is equal to
\[
\int_{\tilde \Gamma\bs \H^d}
\left\langle \overline{\xi_k(f_{m,\mu}(\tau))}, \tilde \Theta_{KM}(\tau,z,h)\right\rangle v^{\kappa}\, d\mu(\tau).
\]
Since the integral of the cusp form $\xi_k(f_{m,\mu})$ against the Eisenstein series $E_L(\tau,\kappa)$ vanishes, we obtain the assertion.

If $n\geq 1$, then one can show by means of Proposition~\ref{prop:ddcex2}
that $dd^c\Phi_{m,\mu}(z,h)$ is equal to the value at $s'=0$ of the holomorphic continuation in $s'$ of
\[
\frac{1}{\sqrt{D}}
\int_{\tilde \Gamma_\infty\bs \H^d}
\left\langle \overline{\delta_k(f_{m,\mu}(\tau))}, \tilde \Theta_{KM}(\tau,z,h)\right\rangle v^{\kappa}\norm(v)^{s'/2}\, d\mu(\tau)-B(m,\mu)\Omega.
\]
Again the assertion follows by unfolding the theta integral.
\end{proof}

\begin{remark}
\label{rem:coeff}
Let $\mu\in L'/L$ and $m\in \partial_F^{-1}+Q(\mu)$ be totally positive.
We have
\[
B(m,\mu)=-\frac{\deg(Z(m,\mu))}{\vol(X_K)}=-\frac{1}{2}A(m,\mu) .
\]
\end{remark}

\begin{proof}
Integrating the identity of \eqref{eq:bokme} (respectively its analogue for $n\geq 1$)
against $\Omega^{n-1}$ we obtain by means of \eqref{eq:kms} that
\[
\int_{X_K}(dd^c\Phi_{m,\mu})\Omega^{n-1}=
-B(m,\mu)\int_{X_K}\Omega^{n}.
\]
On the other hand, Corollary \ref{cor:pl} implies that the left hand side is equal to $\delta_{Z(m,\mu)}(\Omega^{n-1})$.
Consequently, $\deg(Z(m,\mu))= -B(m,\mu)\vol(X_K)$.
(Alternatively, this can be proved as in \cite{Ku:Integrals} Theorem 4.20, using \eqref{eq:kms} and the Thom form property of $\varphi_{KM}$.)
\end{proof}

\subsection{Meromorphic modular forms and special divisors}
We now use Theorem \ref{thm:bokm} to derive an analogue of Borcherds' result on automorphic products (see \cite{Bo2}, Theorem~13.3).


\begin{lemma}
\label{lem:logprod}
Let $U\subset \C^n$ be a convex domain. Let $C$ be an analytic  divisor on $U$, and let $\psi:U\setminus C\to \R$ be a $C^2$-function
with a logarithmic singularity   along $C$. If $\psi$ is pluriharmonic (i.e.~$\partial\bar\partial \psi =0$), then there exists a meromorphic function $\Psi$ on $U$
such that
$\psi=\log|\Psi|$.
\end{lemma}

\begin{proof}

By the assumption on $U$ we have $H^1(U,\calO_U)=H^2(U,\Z)=0$ and the multiplicative Cousin problem is universally solvable.
Hence there is a meromorphic function $G$ on $U$ such that $C=\dv(G)$. The assumption on $\psi$ implies that
\[
\psi-\log|G|
\]
extends to a pluriharmonic real analytic function on $U$.
Since $U$ is simply connected,  there exists a holomorphic function
$H:U\to \C$  such that
\[
\Re(H)= \psi-\log|G|,
\]
see e.g.~\cite{GR} Chapter IX, Section C.
Rewriting this as
\[
\psi=\log|e^H\cdot G|,
\]
we see that we can take $\Psi= e^H\cdot G$.
\end{proof}

The function $\Psi$ in the lemma has divisor $C$. By the maximum modulus principle, it is uniquely determined up to a constant of modulus $1$.


\begin{lemma}
\label{lem:3.21}
(cp.~\cite{Bo2} Lemma 13.1.) Let $r\in \Q$.
 Suppose that $\Psi$ is a meromorphic function on $\calH\times H(\hat \Q)/K$ for which $|\Psi(z,h)|\cdot |y|^{r}$ is invariant under $H(\Q)$.
 Then there exists a unitary multiplier system $\chi:H(\Q)\times H(\hat \Q)\to \C^\times$ of weight $r$ such that $\Psi$ is a
 meromorphic modular form of weight $r$, level $K$,  and multiplier system $\chi$.
\end{lemma}

\begin{proof}
The hypothesis implies that for every $\gamma\in H(\Q)$ and $h\in H(\hat \Q)$, the function
\[
\frac{\Psi(\gamma z,\gamma h)}{\Psi(z,h)} j(\gamma,z)^{-r}
\]
is holomorphic on $\D$ and has constant modulus $1$. By the maximum modulus principle it has to be constant, say equal to
$\chi(\gamma,h)$. The right-invariance of $\Psi$ under $K$ implies that $\chi$ is right-invariant under $K$. Moreover, it is easily checked that $\chi$ satisfies the cocycle condition of a multiplier system.
\end{proof}

\begin{theorem}
\label{thm:bop}
Let
\[
f=\sum_{\mu\in L'/L} \sum_{m\gg 0} c(m,\mu)f_{m,\mu}(\tau)\in  M^!_{k,\bar\rho_L}
\]
be a weakly holomorphic Whittaker form with coefficients $c(m,\mu)\in \Z$.
Then there exists a function  $\Psi(z,h,f)$ on $\calH \times H(\hat \Q)$ with the following properties:
\begin{enumerate}
\item[(i)]
$\Psi$ is a meromorphic modular form for $H(\Q)$ of weight $-B(f)$ and level $K$ with a unitary multiplier system of finite order.
\item[(ii)]
The divisor of $\Psi$ is equal to $Z(f)$.
\item[(iii)]
The Petersson metric of $\Psi$ is given by
\[
-\log\|\Psi(z,h,f)\|_{Pet}^2 =
\Phi(z,h,f).
\]
\end{enumerate}
\end{theorem}

\begin{proof}
We use Theorem \ref{thm:bokm}. The assumption that $f$ is weakly holomorphic means that $\xi_k(f)=0$. Consequently, we have
\begin{align*}
dd^c\Phi(z,h,f)&=-B(f)\Omega\\
&=B(f)dd^c\log|y|^2.
\end{align*}
Hence the function $\Phi(z,h,f)-B(f)\log|y|^2$ is pluriharmonic on $\calH$ and has a logarithmic singularity along $-2Z(f)$.
According to Lemma \ref{lem:logprod} there exists a meromorphic function
$\Psi(z,h,f)$ on $\calH\times H(\hat \Q)/K$ such that
\begin{align*}
\Phi(z,h,f)-B(f)\log|y|^2= -2\log|\Psi(z,h,f)|.
\end{align*}
So $\Phi(z,h,f)=-\log\|\Psi(z,h,f)\|_{Pet}^2$, where the Petersson metric is in weight $-B(f)$.
By construction, the divisor of $\Psi$ is equal to $Z(f)$.

The invariance properties of $\Phi(z,h,f)$ and Lemma \ref{lem:3.21} imply that there is a unitary multiplier system
$\chi$ such that $\Psi$ is a meromorphic modular form of weight $-B(f)$, level $K$, and multiplier system $\chi$.

When $n>2$, the Lie group $H(\R)$ has no almost simple factor of real rank $1$.
Hence, according to \cite{Mar} (Proposition 6.19 on p.~333), the multiplier system $\chi$ has finite order.
When $n\leq 2$ we will prove that $\chi$ has finite order in the next section by means of the embedding trick, see Corollary \ref{cor:finmult}.
\end{proof}

\begin{remark}
When $d=1$, then Theorem \ref{thm:bop} is compatible (up to a constant) via Proposition \ref{prop:wm} with the Borcherds lift of weakly holomorphic modular forms \cite[Theorem 13.3]{Bo2}. This follows from \cite{Br2}, Proposition 2.11.
\end{remark}

\section{Modularity of special divisors }

We first assume that $n>2$.
We use Theorem \ref{thm:bop} to prove that the generating series of special divisors is a Hilbert modular form of weight $\kappa$ with values in the Chow group. This result is proved in \cite{YZZ} using the modularity result for the cohomology classes of Kudla--Millson \cite{KM1, KM2, KM3}. Our proof is a variant of the proof that Borcherds gave for $F=\Q$, see \cite{Bo3}.

Next we consider the cases of small dimension $n=1,2$. Following \cite{YZZ} and \cite{Bo2}, the modularity result can be extended to this case by means of the embedding trick. As an application we
prove the finiteness of the multiplier system of the meromorphic modular form $\Psi(z,h,f)$ of Theorem \ref{thm:bop} for $n=1,2$.

\subsection{The case $n>2$}

For a Schwartz function $\varphi\in S_L\subset S(V(\hat F))$
we consider the special divisors $Z(m,\varphi)$.
According to Remark \ref{rem:deg0}, the divisors
\[
Z^0(m,\varphi):=Z(m,\varphi)-\frac{\deg Z(m,\varphi)}{\vol(X_K)} c_1(\calM_1)
\]
have degree $0$.
We define the generating functions
\begin{align*}
A(\tau,\varphi)&=A_V(\tau,\varphi):= -c_1(\calM_1) + \sum_{m\gg0} Z(m,\varphi)q^m,\\
A^0(\tau,\varphi)&= A^0_V(\tau,\varphi):= \sum_{m\gg 0} Z^0(m,\varphi)q^m,
\end{align*}
which we view as formal power series with coefficients in $\ch^1(X_K)$.
Here $q^m:=e(\tr(m\tau))$ and $c_1(\calM_1)$ is the Chern class in $\ch^1(X_K)$ of $\calM_1$  given by the divisor of a rational section.
Moreover, we define
the $\rho_L$-valued generating functions
\begin{align*}
A(\tau)&=  \sum_{\mu\in L'/L} A(\tau,\chi_\mu)\chi_\mu,\\
A^0(\tau)&=  \sum_{\mu\in L'/L} A^0(\tau,\chi_\mu)\chi_\mu.
\end{align*}
When the Eisenstein series $E_L(\tau,\kappa)$ is holomorphic, then
Remark \ref{rem:coeff} implies that
\[
A^0(\tau) = A(\tau)+c_1(\calM_1)E_L(\tau,\kappa).
\]

\begin{theorem}
\label{thm:modularity}
The generating series $A^0(\tau)$ belongs to $S_{\kappa,\rho_L}\otimes \ch^1(X_K)$.\end{theorem}

For the proof of the theorem we need the following  linear algebra lemma.

\begin{lemma}
\label{lem:la}
Let $X$, $Y$ be vector spaces over a field $E$,
and let $\beta:X\times Y\to E$ be a non-degenerate bilinear form.
Let $X_1\subset X$ be a subspace and put
\begin{align*}
X_1^\perp &= \{y\in Y; \; \text{$\beta(x,y) = 0$ for all $x\in X_1$}\},\\
X_1^\perp{}^\perp &= \{x\in X; \; \text{$\beta(x,y) = 0$ for all $y\in X_1^\perp$}\}.
\end{align*}
If $X_1$ is finite dimensional, then $X_1^\perp{}^\perp=X_1$.\hfill $\square$
\end{lemma}


\begin{proof}[Proof of Theorem \ref{thm:modularity}]
We write $P_L$ for the vector space of $\rho_L$-valued formal power series with vanishing constant term,
that is, the vector space of formal power series of the form
\[
g= \sum_{\mu\in L'/L} \sum_{\substack{m\in \partial^{-1}+Q(\mu)\\ m\gg 0}} b(m,\mu)q^m\chi_\mu.
\]
We may view $S_{\kappa,\rho_L}$ as a subspace of $P_L$ by taking the Fourier expansion of a cusp form.
We extend the pairing $\{\cdot,\cdot\}$  between $H_{k,\bar\rho_L}$ and $S_{\kappa,\rho_L}$ defined in \eqref{defpair} to a non degenerate pairing between
$H_{k,\rho_L}$ and $P_L$ using the formula  \eqref{pairalt}.

We apply Lemma \ref{lem:la} with $X=P_L$, $X_1=S_{\kappa,\rho_L}$, $Y=H_{k,\bar \rho_L}$, and the pairing $\{\cdot,\cdot\}$. We have that $M^!_{k,\bar \rho_L}= S_{\kappa,\rho_L}^\perp$.
The generating series $A^0(\tau)$ is an element of $P_L\otimes \ch^1(X_K)$. According to the lemma it belongs to  $S_{\kappa,\rho_L}\otimes \ch^1(X_K)$ if and only if
\begin{align}
\label{eq:key}
\{A^0, f\} = 0\in \ch^1(X_K)_\C
\end{align}
for all $f\in M^!_{k,\bar \rho_L}$.
Adapting the argument of \cite{McG}, it can be proved that $S_{\kappa,\rho_L}$ has a basis of cusp forms with coefficients in $\Z$.
Hence it suffices to verify \eqref{eq:key} for those
\[
f= \sum_\mu\sum_{m\gg 0} c(m,\mu) f_{m,\mu}
\]
which are integral linear combinations of the $f_{m,\mu}$.
So for such $f$ we have to show that
\[
 \sum_\mu\sum_{m\gg 0} c(m,\mu)Z(m,\mu) =  \sum_\mu\sum_{m\gg 0} c(m,\mu)\frac{\deg Z(m,\varphi)}{\vol(X_K)} c_1(\calM_1)\in \ch^1(X_K)_\Q.
\]
In view of Remark \ref{rem:coeff} this is equivalent to
\[
Z(f) = -B(f)c_1(\calM_1)\in \ch^1(X_K)_\Q.
\]
But this relation is exactly produced by Theorem \ref{thm:bop}.
\end{proof}

\subsection{The embedding trick}


Let $V_1\subset V$ be a quadratic subspace defined over $F$, and assume that $V_1$ has signature
\[
((n_1,2),(n_1+2,0),\dots, (n_1+2,0))
\]
with $1\leq n_1 \leq n$. Let $V_2=V_1^\perp$. Then $V_2$ is totally positive definite and $V=V_1\oplus V_2$.
We view $H_1:=\Res_{F/\Q}\GSpin(V_1)$ as a subgroup of $H$ acting trivially on $V_2$ and put
$K_1:= H_1(\hat\Q)\cap K$. We let $\D_{V_1}$ be the sub-Grassmannian of $\D$ given by the oriented negative definite $2$-dimensional subspaces $z$ of $V_{1,\sigma_1}\subset V_{\sigma_1}$. We obtain an embedding of Shimura varieties
\[
\iota: X_{K_1,V_1}:= H_1(\Q)\bs (\D_{V_1}\times H_1(\hat \Q))/K_1 \longrightarrow X_{K,V}.
\]
It induces a pull-back homomorphism of the Chow groups
\[
\iota^*: \ch^1(X_{K,V})\longrightarrow \ch^1(X_{K_1,V_1}).
\]

The pull back of the generating series $A_V(\tau,\varphi)$ is computed in \cite{YZZ}, Proposition 3.1. Let $\varphi_i\in S(V_i(\hat F))$ and assume that $X_{K_1,V_1}$ is compact. Then
\begin{align}
\label{eq:pullback}
\iota^*(A_V(\tau,\varphi_1\otimes \varphi_2))= A_{V_1}(\tau,\varphi_1)\cdot \theta_{S,V_2}(\tau;\varphi_2),
\end{align}
where
\[
\theta_{S,V_2}(\tau;\varphi_2)= \sum_{\lambda\in V_2(F)}\varphi_2(\lambda)e(\tr Q(\lambda)\tau)
\]
is the usual theta series of the positive definite quadratic space $V_2$.
By embedding quadratic spaces over $F$ of dimension $n=1$ or $2$ into larger spaces, employing the pull back-formula \eqref{eq:pullback}, and varying $\varphi_2$, one obtains (see \cite{YZZ}, proof of Theorem 1.3):

\begin{proposition}
\label{prop:modularity}
Theorem \ref{thm:modularity} also holds for $n=1,2$.
\end{proposition}

\begin{corollary}
\label{cor:finmult}
Let
$
f\in  M^!_{k,\bar\rho_L}
$
be a weakly holomorphic Whittaker form as in Theorem~\ref{thm:bop} and let $\Psi(z,h,f)$ be the corresponding meromorphic modular form for $H(\Q)$. The multiplier system of $\Psi$ has also finite order when $n=1,2$.
\end{corollary}

\begin{proof}
Since $f$ is weakly holomorphic and $A^0\in S_{\kappa,\rho_L}\otimes \ch^1(X_K)$ by
Proposition \ref{prop:modularity}, we have
$\{A^0,f\}=0\in \ch^1(X_K)_\Q$.
This is equivalent to
\[
Z(f) = -B(f)c_1(\calM_1)\in \ch^1(X_K)_\Q.
\]
So there is a meromorphic modular form $\tilde\Psi$ of weight $-B(f)$ and level $K$ with a multiplier system of {\em finite} order such that
$\dv(\tilde \Psi) = Z(f)$.
Consequently, $\Psi/\tilde \Psi$ is a holomorphic modular form of weight $0$ with no zeros and a multiplier system of possibly infinite order. But it is easily seen that such a modular form must be constant.
This proves the corollary.
\end{proof}

\section{Examples}
\label{sect:8}

Here we give some examples illustrating Theorem \ref{thm:bop} and Theorem \ref{thm:modularity}.

\subsection{Shimura curves}

Let $B/F$ be a quaternion algebra which is split at $\sigma_1$ and ramified at all the other real places of $F$.
Let $\delta\in \partial_F^{-1}$ be an element such that $\sigma_i(\delta)>0$ for $i=2,\dots,d$. We write $\norm(x)$ for the reduced norm of $x\in B$.
Then $(B, \delta\norm)$ is a quadratic space over $F$ of signature $((2,2),(4,0),\dots,(4,0))$.
The group $\GSpin(B)$ can be computed using similar arguments as in \cite{KuRa} \S0. One finds that
\[
\GSpin(B)\cong \{(g_1,g_2)\in B^\times \times  B^\times ;\; \norm(g_1)=\norm(g_2)\}.
\]
Under this identification the natural action of $\GSpin(B) $ on $B$ is identified with
$(g_1,g_2).b = g_1 b g_2^{-1}$ for $b\in B$.
We may view an $\calO_F$-order $\calO\subset B$ as an $\calO_F$-lattice in the quadratic space $B$.
The Shimura variety $X_K$ can be viewed as the product of two Shimura curves or as product of a Shimura curve with itself depending on the choice of the compact open subgroup $K$.

The subspace $B^0\subset B$ of trace zero elements of $B$ defines a quadratic subspace.  When $\sigma_1(\delta)$ is also positive, it has signature
$((1,2),(3,0),\dots,(3,0))$. We may identify
\[
\GSpin(B^0)\cong B^\times,
\]
the action on $B^0$ being given by conjugation.
So Theorem \ref{thm:bop} gives rise to automorphic forms on Shimura curves over $F$ with known divisor supported on special divisors.

It is interesting to consider the examples of Shimura curves investigated in \cite{El} in that way.
Here we briefly discuss the Shimura curve $X$ associated to the triangle group $G_{2,3,7}$, see \cite{El} Section 5.3. It is a genus zero curve with a number of striking properties. For instance, the minimal quotient area of a discrete subgroup of $\operatorname{PSL}_2(\R)$ is $1/42$, and it is only attained by the triangle group $G_{2,3,7}$.
Let $F$ be the totally real cubic field $\Q(\cos(2\pi/7))=\Q[x]/(x^3+x^2-2x-1)$.
It has discriminant $49$ and class number $1$. The inverse different $\partial_F^{-1}$ has a generator $\delta$ such that $\sigma_i(\delta)>0$ for $i=2,3$.
Let $B$ be the quaternion algebra over $F$ which is ramified at exactly the real places $\sigma_2,\sigma_3$. We view $(B,\delta\norm)$ as a quadratic space
over $F$ as above. Let $L$ be a maximal $\calO_F$-order of $B$.  Then at all finite places $\frakp$ of $F$ the lattice $L_\frakp$ is isomorphic to the $2\times 2$ matrices with entries in $\calO_{F,\frakp}$. This implies that $L$ is even unimodular and the corresponding Weil representation $\rho_L$ is trivial. The underlying lattice $(L,Q_\Q)$ over $\Z$ is isometric to $E_8\oplus H\oplus H$, where $H$ denotes a hyperbolic plane over $\Z$. We let $K\subset H(\hat\Q)$ be the stabilizer of the lattice $\hat L$. Then the corresponding Shimura variety $X_K$ is isomorphic to $X\times X$, where $X$ is the quotient of $\H$ by the group of units of norm $1$ of $L$.

The Jacquet--Langlands correspondence provides an isomorphism between the space of Hilbert cusp forms $S_\kappa$ of parallel weight $\kappa=(2,2,2)$ for $\Sl_2(\calO_F)$ and the space of holomorphic differential 1-forms on $X$. Since $X$ has genus $0$, we see that $S_\kappa=\{0\}$.
Consequently, any harmonic Whittaker form $f$ of weight $k=(0,2,2)$ is weakly holomorphic. Its regularized theta lift gives rise to a meromorphic modular form of weight $-B(f)$ on $X_K$. Its divisor $Z(f)$ is a linear combination of Hecke correspondences. The generating series $A(\tau)$ of special divisors is equal to
\[
c_1(\calM_1)E_L(\tau,\kappa).
\]

For any totally positive $m\in \partial_F^{-1}$ there exists a holomorphic modular form $\Psi_m$ on $X_K$ with divisor $Z(m,0)$ and weight $B(m,0)$, given by Theorem
\ref{thm:bop}. These modular forms are analogues of the form $j(z_1)-j(z_2)$ and its multiplicative Hecke translates on $Y(1)\times Y(1)$, where $j$ denotes the classical $j$-function and $Y(1)=\Sl_2(\Z)\bs \H$. In view of the work of Gross and Zagier on singular moduli, it would be interesting to compute the CM values of the functions $\Psi_m$.

Similar unimodular lattices can be constructed over
any totally real field of odd degree for which the different is a principal ideal in the narrow sense.
But clearly the space of cusp forms for $\Sl_2(\calO_F)$ of parallel weight $2$ will be non-trivial in general.

\subsection{Even unimodular lattices over real quadratic fields}

\label{sect:8.2}

It would be interesting to have some existence or classification results for even unimodular $\calO_F$-lattices (in the sense of Section~\ref{sect:lattices}),
since the Weil representation $\rho_L$ is trivial in this case. Adapting the arguments of \cite{Scha} and \cite{Ch}, one can obtain some results in this direction. (Notice that the setup in these references is slightly different from ours. For instance, they consider $\calO_F$-valued quadratic forms and define the dual lattice as the $\calO_F$-dual.)

Here we briefly consider the case where $F$ is a real quadratic field of discriminant $D$.

\begin{lemma}
There exists an even unimodular $\calO_F$-lattice of signature $((n,2),(n+2,0))$ if and only if
$n$ is divisible by $4$ and $\calO_F$ contains a totally positive unit $\eps$ such that $-\eps $ is a square modulo $4\calO_F$.
\end{lemma}

\begin{proof}
A lattice $(L,Q)$ is even unimodular of signature $((n,2),(n+2,0))$ in the sense of Section~\ref{sect:lattices}, if and only if $(L,\sqrt{D}Q)$ is a unimodular lattice of signature $((n,2),(0,n+2))$ in the sense of \cite{Ch}. Hence the lemma follows from Theorem~3 in \cite{Ch}.
\end{proof}

\begin{remark}
i) If $F=\Q(\sqrt{a})$ with $a>0$ squarefree and $a\equiv 3\pmod 4$ then the condition of the lemma is always fulfilled when $n$ is divisible by $4$. In this case we can take $\eps=1$.

ii) If the fundamental unit of $F$ has norm $-1$, then the condition of the lemma is never fulfilled, since every totally positive unit is a square and therefore $-1$ would have to be a square modulo $4\calO_F$. This is not the case as an elementary computation shows.

iii) If $\eps\in \calO_F$ is a totally positive unit such that
$-\eps = \alpha^2 -4\beta$ with $\alpha,\beta\in \calO_F$, then the lattice $L_1$ given by the Gram matrix
\begin{align*}
\frac{-1}{\sqrt{D}}
\begin{pmatrix}
2 & \alpha\\
\alpha & 2\beta
\end{pmatrix}
\end{align*}
is even unimodular of signature $((0,2),(2,0))$.
\end{remark}

It also follows from Theorem~3 in \cite{Ch} that for
any real quadratic field $F$ there
exists an even unimodular $\calO_F$-lattice $L_0$ of signature $((4,0),(4,0))$.
The corresponding lattice over $\Z$ is isometric to $E_8$. One can construct such a lattice explicitly by modifying the construction of \cite{Scha} Section 3.
For instance, for $F=\Q(\sqrt{3})$,
one can take the lattice with the Gram matrix
\begin{align*}
\begin{pmatrix}
-\sqrt{3}+2 & \frac{\sqrt{3}-1}{2} & \frac{-\sqrt{3}+3}{6} & \frac{\sqrt{3}}{6}\\
\frac{\sqrt{3}-1}{2} & 3 & \frac{\sqrt{3}}{2} & \frac{-\sqrt{3}-3}{6}\\
\frac{-\sqrt{3}+3}{6}  &\frac{\sqrt{3}}{2}  & 1 & \frac{\sqrt{3}+1}{2} \\
\frac{\sqrt{3}}{6} &\frac{-\sqrt{3}-3}{6}  & \frac{\sqrt{3}+1}{2} & 2+\sqrt{3}
\end{pmatrix}.
\end{align*}
If $n$ is divisible by $4$ and $F$ satisfies the condition of the lemma, then  $L=L_0^{\oplus n/4}\oplus L_1$ is even unimodular of signature $((n,2),(n+2,0))$.

Let now $F=\Q(\sqrt{3})$ and let $L$ be the lattice $L_0\oplus L_1$ as above. Let $K\subset H(\hat\Q)$ be the stabilizer of $\hat L$.
One can show that the space of Hilbert cusp forms $S_\kappa$ of parallel weight $\kappa=(3,3)$ for $\Sl_2(\calO_F)$ vanishes.
Consequently, any harmonic Whittaker form $f$ of weight $k=(-1,3)$ is weakly holomorphic.
For any totally positive $m\in \partial_F^{-1}$ there exists a holomorphic modular form $\Psi_m$ on $X_K$ with divisor $Z(m,0)$ and weight $B(m,0)$, given by Theorem
\ref{thm:bop}. The generating series $A(\tau)$ of special divisors is equal to
$c_1(\calM_1)E_L(\tau,\kappa)$.

\end{document}